\documentclass{amsart}

\usepackage{amsmath}
\usepackage{amssymb}
\usepackage{verbatim}
\usepackage{enumerate}
\usepackage[all]{xy}

\numberwithin{equation}{section}
\newtheorem{thm}[equation]{Theorem}
\newtheorem{prop}[equation]{Proposition}
\newtheorem{defn}[equation]{Definition}
\newtheorem{lemma}[equation]{Lemma}
\newtheorem{remark}[equation]{Remark}
\newtheorem{ex}[equation]{Example}
\newtheorem{coro}[equation]{Corollary}

\newcommand{\exref}[1]{Ex\-am\-ple \ref{#1}}
\newcommand{\thmref}[1]{Theo\-rem \ref{#1}}
\newcommand{\defref}[1]{Def\-i\-ni\-tion \ref{#1}}

\newcommand{\lemref}[1]{Lem\-ma \ref{#1}}
\newcommand{\propref}[1]{Prop\-o\-si\-tion \ref{#1}}
\newcommand{\corref}[1]{Cor\-ol\-lary \ref{#1}}

\newcommand{\remref}[1]{Re\-mark \ref{#1}}

\providecommand{\normaleq}{\unlhd}
\providecommand{\normal}{\lhd}

\providecommand{\union}{\cup}

\providecommand{\intersect}{\cap}

\DeclareMathOperator{\chr}{char }

\DeclareMathOperator{\rad}{rad }

\DeclareMathOperator{\End}{End }
\DeclareMathOperator{\Sym}{Sym }

\DeclareMathOperator{\Adj}{Adj }
\DeclareMathOperator{\GL}{GL}
\DeclareMathOperator{\SL}{SL}
\DeclareMathOperator{\Sp}{Sp}
\DeclareMathOperator{\GO}{GO}
\DeclareMathOperator{\GU}{GU}

\DeclareMathOperator{\Aut}{Aut}
\DeclareMathOperator{\Inv}{Inv}
\DeclareMathOperator{\Isom}{Isom}
\DeclareMathOperator{\im}{im }

\DeclareMathOperator{\Diag}{Diag }
\DeclareMathOperator{\disc}{disc }
\DeclareMathOperator{\spec}{spec }
\newcommand{\Bi}{\mathsf{Bi} }
\newcommand{\Grp}{\mathsf{Grp} }

\begin{document}




\title{Decomposing $p$-groups via Jordan algebras}
\author{James B. Wilson}
\thanks{This research was supported in part by NSF Grant DMS 0242983.}
\date{\today}
\email{jwilson7@uoregon.edu}
\address{
	Department of Mathematics,
	University of Oregon,
	Eugene, Oregon 97403.
}

\begin{abstract}
For finite $p$-groups $P$ of class $2$ and exponent $p$ the following are invariants
of fully refined central decompositions of $P$: the number of members in the decomposition,
the multiset of orders of the members, and the multiset of orders of their centers.  Unlike for direct product decompositions, $\Aut P$ is not always transitive on the set of fully refined central decompositions, and the number of orbits can in fact be any positive integer.  

The proofs use the standard semi-simple and radical structure of Jordan algebras.
These algebras also produce useful criteria for a $p$-group to be centrally indecomposable.
\end{abstract}



\maketitle

\tableofcontents

%
%
\section{Introduction}\label{sec:intro}

A \emph{central decomposition} of a group $G$ is a set $\mathcal{H}$ of subgroups in which distinct members commute, and $G$ is generated by $\mathcal{H}$ but no proper subset.  A group is \emph{centrally indecomposable} if its only central decomposition consists of the group itself.  A central decomposition is \emph{fully refined} if it consists of centrally indecomposable subgroups.

We prove:
\begin{thm}\label{thm:central-count}
For $p$-groups $P$ of class $2$ and exponent $p$,
\begin{enumerate}[(i)]
\item the following are invariants of fully refined central decompositions of $P$:
the number of members, the multiset of orders of the members, and the multiset of orders of the centers of the members; and
\item the number of $\Aut P$-orbits acting on the set of fully refined central decompositions can be any positive integer.
\end{enumerate}
\end{thm}

Central decompositions arise from, and give rise to, central products (cf. 
Section \ref{sec:cent-prod}), and hence \thmref{thm:central-count}.$(i)$ is a theorem
of Krull-Remak-Schmidt type (cf. \cite[(3.3.8)]{Robinson}).  That theorem states that the multiset of isomorphism types of fully refined direct decompositions (Remak-decompositions) is uniquely determined by the group, and the automorphism group is transitive on the set of Remak-decompositions.  \thmref{thm:central-count}.$(ii)$ points out how unrelated the proof of \thmref{thm:central-count}.$(i)$ is to that of the classical Krull-Remak-Schmidt theorem.  Moreover, inductive proofs do not work for central decompositions.  For example, a quotient by a member in a central decomposition generally removes the subtle intersections of other factors and so is of little use.  Similarly, automorphisms of a member in a central decomposition usually do not extend to automorphisms of the entire group. 

We conjecture that under the hypotheses of \thmref{thm:central-count}, even the multiset of isomorphism types of a fully refined central decomposition of $P$ is uniquely determined by $P$.  For details see Section \ref{sec:conjecture}.

While the literature on direct decompositions is vast, little appears to have been
done for central decompositions.  For $p$-groups, results similar to \thmref{thm:central-count} have concentrated on central decompositions with centrally indecomposable subgroups of rank 2 and 3, with various constraints on their centers \cite{Abbasi:2,Abbasi:1,Tang:cent-1,Tang:cent-2}.  Using entirely different techniques, our setting applies to groups of arbitrary rank at the cost of assuming exponent $p$.

The methods used in this paper involve bilinear maps and non-associative algebras, but not the nilpotent Lie algebras usually associated with $p$-groups.  We introduce a $*$-algebra and a Jordan algebra in order to study central decompositions.  The approach leads to a many other results for $p$-groups and introduces a surprising interplay between $p$-groups, symmetric bilinear forms, and various algebras.  Most of these ideas are developed in subsequent works. As the algebras we use are easily computed, in \cite{Wilson:algo-cent} we provide algorithms for finding fully refined central decompositions and related decompositions -- even for $p$-groups of general class and exponent.  In \cite{Wilson:grp-alge} we prove there are $p^{2n^3/27 + Cn^2}$ centrally indecomposable groups of order $p^n$, which is of the same form as the Higman-Sims bound on the total number of groups of order $p^n$ \cite{Higman:enum,Sims:enum}.  In \cite{Wilson:grp-alge} we also prove that a \emph{randomly presented} group of order $p^n$ is centrally indecomposable, and we characterize various \emph{minimal} centrally indecomposable $p$-groups by means of locally finite $p$-groups, including those $p$-groups with $P'\cong \mathbb{Z}_p^2$.  Finally, in \cite{Wilson:isoclinism} we address central decompositions of $2$-groups, $p$-groups of arbitrary exponent, and $p$-groups of arbitrary class, by means of an equivalence on $p$-groups related to the isoclinism of P. Hall \cite{Hall:isoclinism}.

		%
		%
\subsection{Outline of the proof}

Section \ref{sec:central} contains background and notation for central decompositions 
of groups and orthogonal decompositions of bilinear maps.  

Section \ref{sec:p-bilinear} translates $p$-groups $P$ of class 2 and exponent $p$ into alternating bilinear maps on $P/P'$ induced by commutation.  This approach is well-known
and appears as early as Baer's work \cite{Baer:class-2} and refined in
\cite{Kaloujnine:class-2} and \cite{Warfield:nil}; however, such techniques have
been upstaged by appealing to various associated Lie algebras of Kaloujnine,  Lazard, Mal'cev and others \cite{Khukhro}.  By contrast, the bilinear approach translates unwieldy central decompositions into natural-looking orthogonal decompositions, and automorphisms into pseudo-isometries (\thmref{thm:main-reduction}).  

In Section \ref{sec:adj-sym} we introduce two algebraic invariants of bilinear maps: the associative $*$-algebra of adjoint operators, and the Jordan algebra of self-adjoint operators.  The first of these encodes isometries, while the second encodes orthogonal decompositions via sets of pairwise orthogonal idempotents (\thmref{thm:perp-idemp}).   We use these algebras to give criteria for indecomposable bilinear maps and centrally indecomposable $p$-groups (\corref{coro:indecomp} and \thmref{thm:indecomps}).  We also prove the first part of \thmref{thm:central-count}.$(i)$.
  
In Section \ref{sec:main} we prove that a certain subgroup of isometries acts on 
suitable sets of idempotents of our Jordan algebra with the same orbits as the full isometry 
group.  Using the radical theory of Jordan algebras and the classification of finite 
dimensional simple Jordan algebras we identify the orbits of the isometry group acting 
on the set of fully refined orthogonal decompositions (and therefore the orbits of 
$C_{\Aut P}(P')$ on the set of fully refined central decompositions of $P$) 
(\corref{coro:main}).  

In Section \ref{sec:orbits}, semi-refined central decompositions are introduced.  These
are derived from properties of symmetric bilinear forms and then interpreted in the 
setting of $p$-groups, leading to the proof of \thmref{thm:central-count}.$(i)$.

The remaining Section \ref{sec:ex} proves \thmref{thm:central-count}.$(ii)$.  We also build families of centrally indecomposable groups of the types in \thmref{thm:indecomps}.  These examples are only a sample of the known constructions of this sort and the proofs provided are self-contained versions of broader results in \cite{Wilson:grp-alge}.

Section \ref{sec:closing} has concluding remarks. 

%
%
\section{Background}\label{sec:central}

Unless stated otherwise, all groups, algebras, and vector spaces will be finite and
$p$ will be odd.  We begin with brief introductions 
to central products and central decompositions of groups, followed by orthogonal 
decompositions of bilinear maps.

		%
		%
\subsection{Central decompositions and products}\label{sec:cent-prod}

Let $\mathcal{H}$ be a central decomposition of a group $G$ (cf. Section \ref{sec:intro}).  The condition $[H,K]=1$ for distinct $H, K\in\mathcal{H}$ shows that $H\intersect \langle \mathcal{H}-\{H\}\rangle\leq Z(G)$ for all $H\in\mathcal{H}$.  Then, the members of $\mathcal{H}$ are normal subgroups of $G$.  

Central decompositions can be realized by means of central products. 
Fix a set $\mathcal{H}$ of groups and a subgroup
$N$ of $\widetilde{\mathcal{H}}:=\prod_{H\in \mathcal{H}} H$ such that
$N\intersect H=1$ for all $H\in\mathcal{H}$.  The \emph{central product} 
of $\mathcal{H}$ with respect to $N$ is $\widetilde{\mathcal{H}}/N$.  
If $\mathcal{H}$ is a central decomposition of a group $G$, then define
$\pi:\widetilde{\mathcal{H}}\to G$ by 
$(x_H)_{H\in\mathcal{H}}\mapsto \prod_{H\in\mathcal{H}} x_H$.
Then $G\cong \widetilde{\mathcal{H}}/\ker \pi$.
These two treatments are equivalent \cite[(11.1)]{Aschbacher}.

In an arbitrary central decomposition $\mathcal{H}$ of a group $G$, in general 
$H\intersect K$ and $H\intersect J$ are distinct, for distinct 
elements $H,K,J\in\mathcal{H}$.  
\begin{defn}
Given a subgroup $M\leq G$ and a central decomposition $\mathcal{H}$ of $G$, 
we call $\mathcal{H}$ an $M$-\emph{central decomposition} if 
$M=H\intersect K$ for all distinct $H,K\in\mathcal{H}$.  The associated central
product is an \emph{$M$-central product}.
\end{defn}

Every central decomposition induces a $Z(G)$-central decomposition 
$\mathcal{H}Z(G)=\{HZ(G):H\in\mathcal{H}\}$.  Some authors write 
$H_1*\cdots *H_s$ or $H_1\circ \cdots \circ H_s$ for a $Z(G)$-central product.
These notations still depend on the given $N\leq H_1\times\cdots\times H_s$.
We require a precise meaning in the following specific case:
\begin{equation}\label{eq:canonical-cent}
	\overbrace{H\circ \cdots \circ H}^n  = \overbrace{H\times\cdots \times H}^n/N
\end{equation}
where $N:=\langle (1,\dots,\overset{i}{x},1,\dots,\overset{j}{x^{-1}},1,\dots)
			| 1\leq i<j\leq n, x\in Z(H)\rangle$.

		%
		%
\subsection{Central decompositions of $p$-groups of class $2$ and exponent $p$}

Using standard group theory, we show that central decompositions of a finite $p$-group $P$ 
of class $2$ and exponent $p$ reduce to central decompositions of a subgroup $Q$ where
$P'=Q'=Z(Q)$ and $P=QZ(P)$.  Furthermore, we show that for our purposes we may 
consider only $Z(Q)$-central decompositions (cf. \corref{coro:refine}).  

\begin{defn}\label{def:central-auto}
An automorphism $\varphi\in \Aut P$ is \emph{upper central} if $Z(P)x\varphi=Z(P)x$, for all $x\in P$,  and \emph{lower central} if $P'x\varphi=P'x$, for all $x\in P$.  The group of
upper central automorphisms we denote by $\Aut_{\zeta} P$ and the lower central
automorphisms by $\Aut_{\gamma} P$.
\end{defn}
As $P$ has class $2$, $\Aut_{\gamma} P\leq \Aut_{\zeta} P$.  Furthermore, every $\alpha\in\Aut_{\gamma} P$ is also the identity on $P'$.

\begin{lemma}\label{lem:contract-1}
\begin{enumerate}[(i)]
\item There are subgroups $Q$ and $A$ of $P$ such that $Z(Q)=Q'=P'$, $A\leq Z(P)$
and $P=Q\times A$.
\item Given subgroups $Q$ and $R$ of $P$ such that $Z(Q)=Q'=P'=R'=Z(R)$ and 
$P=QZ(P)=RZ(P)$, if $A$ is a complement to $Q$ as in $(i)$ then it is also a complement to $R$ so that $P=Q\times A=R\times A$.  Furthermore, there is an upper central automorphism of $P$ sending $Q$ to $R$ and identity on $Z(P)$.
\end{enumerate}
\end{lemma}
\begin{proof}
$(i)$. Since $P/P'$ is elementary abelian, there is $P'\leq Q\leq P$ such that 
$Q\intersect Z(P)=P'$ and $P=QZ(P)$.
Furthermore, $P'=[QZ(P),QZ(P)]=Q'$ and $[P,Z(Q)]=[QZ(P),Z(Q)]=1$,
so $Q'\leq Z(Q)\leq Q\intersect Z(P)=Q'$.

Also, $Z(P)$ is elementary abelian, so there is a complement $A$ to
$P'$ in $Z(P)$.  Whence, $P=QZ(P)=QA$ and 
$Q\intersect A\leq Q\intersect Z(P)\intersect A=P'\intersect A=1$.  
As $A$ is central in $P$, $P=Q\times A$.

$(ii)$. Fix two subgroups $Q$ and $R$ as described in the hypothesis.  So there is a complement $A$ to $Q$ as in $(i)$.  Since $Q\intersect Z(P)=P'=R\intersect Z(P)$ it follows that $P=Q\times A=R\times A$.  Let $\pi:P\to P$ be the projection of $P$ to $R$ with kernel $A$.  Restricting $\pi$ to $Q$ gives a homomorphism $\alpha:Q\to R$.  Furthermore, $P=QA$ so $\alpha$ is surjective, and $Q\intersect A=1$ so $\alpha$ is injective.  Hence $\alpha$ is an isomorphism.  Indeed, $Q'=P'=R'$ and $\pi$ is the identity on $R$, so $\alpha$ is the identity on $Q'=R'$.  Then $\beta=\alpha\times 1_A:Q\times A\to R\times A$ is a upper central automorphism of $P$ sending $Q$ to $R$.
\end{proof}

\begin{defn}
If $\mathcal{H}$ is a central decomposition of $P$, then define
$Z(\mathcal{H})=\{H\in \mathcal{H}:H\leq Z(P)\}$.
\end{defn}

\begin{lemma}\label{lem:contract-2}
Let $\mathcal{H}$ be a fully refined central decomposition of $P$.
If $Q=\langle \mathcal{H}-Z(\mathcal{H})\rangle$ and 
$A=\langle Z(\mathcal{H})\rangle$,
then $P=Q\times A$, $Q'=Z(Q)$ and $Q'A=Z(P)$.  
\end{lemma}
\begin{proof}
Certainly $A\leq Z(P)$ and $P=QA$.  Also $P'=Q'$ and $Z(P)=Z(Q)A$.
As $\mathcal{H}$ is fully refined, every $H\in \mathcal{H}-\mathcal{A}$ is 
centrally indecomposable and so also directly indecomposable.  By
\lemref{lem:contract-1} it follows that $H'=Z(H)$, for all 
$H\in\mathcal{H}-Z(\mathcal{H})$.
As a result, $Q'=Z(Q)$.  Thus  $P=Q\times A$.
\end{proof}

\begin{defn}
Two central decompositions $\mathcal{H}$ and $\mathcal{K}$ of a group $G$ 
are \emph{exchangeable} if, for each $\mathcal{J}\subseteq \mathcal{H}$, 
there is an $\alpha\in \Aut G$ such that $\mathcal{J}\alpha\subseteq \mathcal{K}$
and $(\mathcal{H}-\mathcal{J})\alpha=\mathcal{H}-\mathcal{J}$.
\end{defn}

For instance, if $G=H_1\circ\dots \circ H_s=K_1\circ\cdots \circ K_t$ are
exchangeable decompositions, then $s=t$ and, for each $1\leq i\leq s$,
	\[G=H_1\circ\cdots \circ H_i\circ K_{i+1}\circ\cdots\circ K_t.\]
Replacing $\circ$ with $\times$ we recognize this as the usual exchange 
property for direct decompositions.  The Krull-Remak-Schmidt theorem 
states that all fully refined direct decompositions (Remak-decompositions) 
are exchangeable \cite[(3.3.8)]{Robinson}.  In light of \thmref{thm:central-count}.$(ii)$, a general $p$-group of class 2 and exponent $p$ will have fully refined central decompositions which are not exchangeable.

Subgroups in $Z(\mathcal{H})$ can only be exchanged with subgroups in
$Z(\mathcal{K})$, and similarly for the complements of these sets.

\begin{lemma}
If $\mathcal{H}$ and $\mathcal{K}$ are two fully refined central decompositions of $P$ such that $\mathcal{H}-Z(\mathcal{H})=\mathcal{K}-Z(\mathcal{K})$, then $\mathcal{H}$ and $\mathcal{K}$ are exchangeable.
\end{lemma}
\begin{proof}
Set $Q=\langle \mathcal{H}-Z(\mathcal{H})\rangle$, $A=\langle Z(\mathcal{H})\rangle$, $R=\langle \mathcal{K}-Z(\mathcal{K})\rangle$ and $B=\langle Z(\mathcal{K})\rangle$. By \lemref{lem:contract-1}.$(i)$ it follows that $P=Q\times A=R\times B$ and by \lemref{lem:contract-1}.$(ii)$, $P=Q\times B$ as well.  The projection endomorphism $\pi$ from $P$ to $B$ with kernel $Q$ makes $\alpha=1_Q\times \pi$ an automorphism sending $A$ to $B$ and identity on $Q$.  Since $A$ and $B$ are abelian, any fully refined central decomposition is a direct decomposition so $(Z(\mathcal{H}))\alpha$ is exchangeable with $Z(\mathcal{K})$ by automorphisms of $B$.  As $\Aut B$ extends to $\Aut P$ inducing the identity on $Q$, it follows that $\mathcal{H}$ and $\mathcal{K}$ are exchangeable.
\end{proof}

\begin{thm}\label{thm:refine}
If $\mathcal{H}$ and $\mathcal{K}$ are two fully refined central 
decompositions of $P$ such that $\mathcal{H}Z(P)=\mathcal{K}Z(P)$, then 
$\mathcal{H}$ and $\mathcal{K}$ are exchangeable.
\end{thm}
\begin{proof}
It suffices to prove that a single subgroup of $\mathcal{H}$ can be exchanged with one
in $\mathcal{K}$.  Let $M=Z(P)$ and fix $H\in\mathcal{H}-Z(\mathcal{H})$.  As $\mathcal{H}M=\mathcal{K}M$ there is a $K\in\mathcal{K}$ such that $HM=KM$.  Since $H$ is not contained in $Z(P)$ neither is $K$.  If $J\in\mathcal{K}$ such that $HM=JM$ then
$J\leq \langle K,M\rangle$, and so $\mathcal{K}-\{J\}$ generates $P$.  As $\mathcal{K}$
is fully refined this cannot occur.  So $K$ is uniquely determined by $H$.

By \lemref{lem:contract-1}.$(i)$ and the assumption that $\mathcal{H}$ 
and $\mathcal{K}$ are fully refined, it follows that $H'=Z(H)$ and $K'=Z(K)$.
As $Z(HM)=M=Z(KM)$ it follows that $H Z(HM)=HM=KM=K Z(KM)$.  So by 
\lemref{lem:contract-1}.$(ii)$ there is an automorphism $\alpha$ of $HM=KM$ which 
is the identity on $M$ and maps $H$ to $K$.
Extend $\alpha$ to $P$ by defining 
$\alpha$ as the identity on all $J\in \mathcal{H}-\{H\}$.  This extension exchanges 
$H$ and $K$.
\end{proof}

\begin{coro}\label{coro:refine}
Let $P$ be a $p$-group of class $2$ and exponent $p$.
\begin{enumerate}[(i)]
\item $\Aut_{\zeta} P$ is transitive on Remak-decompositions. 
\item Given two fully refined central decomposition $\mathcal{H}$ and $\mathcal{K}$ of $P$, there is a $\varphi\in \Aut_{\zeta} P$ such that $\mathcal{H}\varphi=\mathcal{K}$ if, and only if, $\mathcal{H}Z(P)=\mathcal{K}Z(P)$.
\end{enumerate}
\end{coro}
\begin{proof}
$(i)$.  This is the Krull-Remak-Schmidt theorem.

$(ii)$.  Suppose that $\mathcal{H}\varphi=\mathcal{K}$ for some $\varphi\in \Aut_{\zeta} P$.  Given $H\in \mathcal{H}$ set $K:=H\varphi$.  Then $HZ(P)/Z(P)=(HZ(P)/Z(P))\varphi = KZ(P)/Z(P)$ so $HZ(P)=KZ(P)$.  Thus $\mathcal{H}Z(P)=\mathcal{K}Z(P)$.  

For the reverse direction, let $\mathcal{H}Z(P)=\mathcal{K}Z(P)$.  Then by \thmref{thm:refine} there is a $\varphi\in\Aut_{\zeta} P$ sending $\mathcal{H}$ to $\mathcal{K}$.
\end{proof}

		%
		%
\subsection{Bilinear and Hermitian maps, isometries, and pseudo-isometries}
		\label{sec:bilinear}

In this section we introduce terminology and elementary properties for bilinear maps which
we will use frequently.  Throughout, let $V$ and $W$ be vector spaces over a 
field $k$.  

A map $b:V\times V\to W$ is \emph{$k$-bilinear} if it satisfies
\begin{equation*}
	b(su+u',tv+v')=stb(u,v)+tb(u',v)+sb(u,v')+b(u',v')
\end{equation*}
for all $u,u',v,v'\in V$ and $s,t\in k$.   Given $X,Y\subseteq V$ define
\begin{equation*}
	b(X,Y) := \langle b(u,v) : u\in X, v\in Y\rangle.
\end{equation*}
For convenience we assume all our bilinear maps have $W=b(V,V)$.  
Whenever $X\leq V$ we can restrict $b$ to
\begin{equation}\label{eq:b-X}
	b_X:X\times X\to b(X,X).
\end{equation}
The \emph{radical} of $b$ is 
\begin{equation*}
	\rad b := \{u\in V:b(u,V)=0=b(V,u)\}.
\end{equation*}
If $\rad b=0$ then $b$ is \emph{non-degenerate}. 
A $k$-bilinear map $b:V\times V\to W$ is called $\theta$-\emph{Hermitian} if
$\theta\in \GL(W)$ and 
\begin{equation}\label{eq:herm}
	b(u,v)=b(v,u)\theta,\qquad \forall u,v\in V.
\end{equation}
As $W=b(V,V)$, $\theta$ is an \emph{involution} (which in this paper will mean 
$\theta^2=1$ and allow $\theta=1$).  Furthermore, $\theta$ is uniquely 
determined by $b$ (assuming $W\neq 0$) and so it is sufficient to say $b$ is Hermitian.  

If $\theta=1_W$ we say that $b$ is \emph{symmetric} and if $\theta=-1_W$ we
call $b$ \emph{skew-symmetric}.  As we work in odd characteristic it follows that
every skew-symmetric bilinear map is equivalently \emph{alternating} in the sense
that $b(v,v)=0$ for all $v\in V$.

Given two $k$-bilinear maps $b:V\times V\to W$ and  $b':V'\times V'\to W'$ a 
\emph{morphism} from $b$ to $b'$ is a pair $(\alpha,\beta)$ 
of linear maps $\alpha:V\to V'$ and $\beta:W\to W'$ such that
\begin{equation}\label{eq:pseudo-def}
	b'(u\alpha,v\alpha)=b(u,v)\beta,\qquad \forall u,v\in V.
\end{equation}
When $\alpha$ is surjective it follows that $W'= b'(V\alpha,V\alpha)$; so,
$\beta$ is uniquely determined by $\alpha$.  In this case we often write $\hat{\alpha}$ 
for $\beta$.  If $\alpha$ and $\hat{\alpha}$ are isomorphisms then we say $b$ and $b'$ are
\emph{pseudo-isometric}.   The term \emph{isometric} is reserved for the special
circumstance where $W=W'$ and $\hat{\alpha}=1_W$.  

The \emph{pseudo-isometry group} is
\begin{equation}\label{eq:isom-*}
\begin{split}
	\Isom^*(b)  := & \{(\alpha,\hat{\alpha})\in\GL(V)\times\GL(W) :\\
		 &  b(u\alpha,v\alpha)=b(u,v)\hat{\alpha}, \forall u,v\in V\},
\end{split}
\end{equation}
and the \emph{isometry group} is
\begin{equation}
	\Isom(b):=\{\alpha\in \GL(V) : b(u\alpha,v\alpha)=b(u,v), \forall u,v\in V\}.
\end{equation}
(The decision to write the isometry group as a subgroup of $\GL(V)$ rather than
$\GL(V)\times \GL(W)$ is to match with the classical definition of the isometry group
of a bilinear form.)   When $b$ is a bilinear $k$-form (i.e.: $W=k$), the pseudo-isometry group goes by various names, including the group of \emph{similitudes} and the \emph{conformal}
group of $b$.  The following is obvious:
\begin{prop}\label{prop:pseudo-isom}
\begin{enumerate}[(i)]
\item If $(\varphi,\hat{\varphi})$ is a pseudo-isometry from $b$ to $b'$ then
$\Isom^*(b)\cong \Isom^*(b')$ via $(\alpha,\hat{\alpha})\mapsto (\alpha^\varphi,\hat{\alpha}^{\hat{\varphi}})$, and $\Isom(b)\cong \Isom(b')$ via
$\alpha\mapsto \alpha^\varphi$.
\item
If $b:V\times V\to W$ is a bilinear map, then $(\alpha,\hat{\alpha})\mapsto \hat{\alpha}$ is a homomorphism from $\Isom^*(b)$ into $\GL(W)$ with kernel naturally
identified with $\Isom(b)$.
\end{enumerate}
\end{prop}
In light of \propref{prop:pseudo-isom}.$(ii)$ we will view $\Isom(b)$ as a subgroup
of $\Isom^*(b)$ and $\Isom^*(b)/\Isom(b)$ as a subgroup of $\GL(W)$.

		%
		%
\subsection{$\perp$-decompositions}

\begin{defn}\label{def:perp}
Let $b:V\times V\to W$ be a $k$-bilinear map.
\begin{enumerate}[(i)]
\item A set $\mathcal{X}$ of subspaces of $V$ is a \emph{$\perp$-decomposition} of $b$
if: $(a)$ $b(X,Y)=0$ for all distinct $X,Y\in\mathcal{X}$ and $(b)$ 
$V=\langle \mathcal{Y}\rangle$ for $\mathcal{Y}\subseteq \mathcal{X}$ if, and only
if, $\mathcal{Y}=\mathcal{X}$.
\item A subspace $X$ of $V$ is a $\perp$-factor if there is a $\perp$-decomposition 
$\mathcal{X}$ containing $X$.  Furthermore, define
\begin{equation*}
	X^\perp := \langle \mathcal{X}-\{X\}\rangle.
\end{equation*}
\item We say $b$ is \emph{$\perp$-indecomposable} if is has only the trivial
$\perp$-decomposition $\{V\}$.
\item A $\perp$-decomposition $\mathcal{X}$ of $b$ is \emph{fully refined} if
$b_X$ is $\perp$-indecomposable for each $X\in\mathcal{X}$ (cf. \eqref{eq:b-X}).
\end{enumerate}
\end{defn}

When $b$ is Hermitian it is also \emph{reflexive} in the sense that $b(u,v)=0$ if, and only if, $b(v,u)=0$, for $u,v\in V$.  Also, $X^\perp=\{x\in V: b(X,x)=0\}$.

Let $\mathcal{X}$ be a $\perp$-decomposition of $b$ and take $X\in\mathcal{X}$.
For each $x\in X\intersect \langle\mathcal{X}-\{X\}\rangle$ we know 
$b(x,\langle \mathcal{X}-\{X\}\rangle)=0$ and $b(x,X)=0$; thus, $b(x,V)=0$.  
Hence, $X\intersect \langle \mathcal{X}-\{X\}\rangle\leq \rad b$.  
Thus a fully refined $\perp$-decomposition is also a direct decomposition of $V$
(and more generally any $\perp$-decomposition, if the bilinear map is non-degenerate.)

The pseudo-isometry group \eqref{eq:isom-*} acts on the set of all $\perp$-decompositions,
but may not be transitive on the set of all fully refined decompositions.  This
fact can already be seen for symmetric bilinear forms (see \thmref{thm:address}). 

		%
		%
\subsection{Symmetric bilinear forms}\label{sec:sym-forms}

Various parts of our proofs and examples require some classical facts about 
symmetric bilinear forms over finite fields.

Let $K$ be a finite field and $\omega\in K$ a non-square.  
By \cite[p. 144]{Artin:geometry}, every $n$-dimensional non-degenerate
 symmetric bilinear $K$-form is isometric to $d:K^n\times K^n\to K$ defined by
\begin{equation}\label{eq:dot-prod}
	d(u,v) := u D v^t,  \forall u,v\in K^n;
\end{equation} 
where $D$ is $I_n$ or $I_{n-1}\oplus [\omega]$.  If $n$ is odd then these two forms are pseudo-isometric, but they are not pseudo-isometric if $n$ is even.  If $A\in \GL(n,K)$ then 
$d(uA,vA)=u(ADA^t)v^t$.  The \emph{discriminant}  of $d$ is 
\begin{equation}\label{def:disc}
	\disc d\equiv\det D \equiv \det ADA^t \pmod{(K^\times)^2},
\end{equation}
for any $A\in \GL(n,K)$ \cite[(3.7)]{Artin:geometry}.  The discriminant 
distinguishes the two isometry classes of non-degenerate symmetric bilinear forms
of a fixed dimension. 
\begin{lemma}\label{lem:line}
Let $d:K^2\times K^2\to K$ be defined as in \eqref{eq:dot-prod}.
\begin{enumerate}[(i)]
\item If $\disc d=[1]$ then 
$\left(\begin{bmatrix} \alpha & \beta\\ \beta & -\alpha\end{bmatrix}, \omega\right)
\in\Isom^*(d)$, where $\omega=\alpha^2+\beta^2\in K$.
\item If $\disc d=[\omega]$ then 
$\left(\begin{bmatrix} 0 & 1\\ \omega & 0\end{bmatrix},\omega\right)
\in\Isom^*(d)$.
\end{enumerate}
\end{lemma}
\begin{proof} This follows by direct computation. \end{proof}

\begin{prop}\label{prop:dot-isom-*}
Let $d$ be as in \eqref{eq:dot-prod}.
Then (by definition) $\Isom(d)$ is the general orthogonal group $\GO(d)$.  Also,
\begin{enumerate}[(i)]
\item if $n$ is odd then $\Isom^*(d)=
\langle (\alpha,1), (s I_n, s^2) \mid \alpha\in\GO(d), s\in K^\times\rangle$;
hence, $\Isom^*(d)/\Isom(d)\cong (K^\times)^2$;
\item if $n$ is even then $\Isom^*(d)=
\langle (\alpha,1), (s I_n,s^2), (\varphi,\omega) \mid \alpha\in\GO(d), 
s\in K^\times \rangle$ where $\varphi:=\phi\oplus \cdots \oplus \phi\oplus \mu$,
$(\phi,\omega)$ is as in \lemref{lem:line}.$(i)$ and
\begin{enumerate}[(a)]
\item if $\disc d=[1]$ then $(\mu,\omega)$ is as in \lemref{lem:line}.$(i)$; and
\item if $\disc d=[\omega]$ then $(\mu,\omega)$ is as in \lemref{lem:line}.$(ii)$.
\end{enumerate}
In particular, $\Isom^*(d)/\Isom(d)\cong K^\times$.
\end{enumerate}
Therefore, $|\Isom^*(d)|=\varepsilon (q-1)|\GO(d)|$ where $q=|K|$,
$\varepsilon=1/2$ if $n$ is odd, and $\varepsilon=1$ if $n$ is even.
\end{prop}
\begin{proof}
By \propref{prop:pseudo-isom}.$(ii)$ we start knowing $\Isom^*(d)/\Isom(d)\leq K^\times$.  Furthermore, $\Isom^*(d)=\{(A,s)\in \GL(V)\times k^\times : ADA^t = s D\}$.  
Hence, for each $(A,s)\in\Isom^*(d)$ we must have $s^n=(\det A)^2$.  
$(i)$.  If $n$ is odd then $s$ must be a square.  Hence, 
$\Isom^*(d)/\Isom(d)\cong (K^\times)^2$.  As $(s I_n, s^2)\in \Isom^*(d)$ it follows
that $\Isom^*(d)=\langle (\alpha,1), (s I_n, s^2) \mid \alpha\in\GO(d), 
s\in K^\times\rangle$.  $(ii)$.  If  $n$ is even, then $(\varphi,\omega)\in \Isom^*(d)$.  Thus  $\Isom^*(d)/\Isom(d) = \langle s^2, \omega : s\in K^\times\rangle=K^\times$
and $\Isom^*(d)=\langle (\alpha,1), (s I_n,s^2), (\varphi,\omega) \mid \alpha\in\GO(d), 
s\in K^\times \rangle$.
\end{proof}

%
%
\section{Bilinear maps and $p$-groups}\label{sec:p-bilinear}

In this section we translate fully refined central decompositions to 
$\perp$-de\-com\-pos\-i\-tions, automorphisms to pseudo-isometries, and back (\propref{prop:auto-isom} and \thmref{thm:main-reduction}).  

To prove these we describe a well-known method to convert $p$-groups of class $2$ into bilinear maps explored as early as \cite{Baer:class-2}, compare \cite{Kaloujnine:class-2}, and \cite[Section 5]{Warfield:nil}.  The method has ties to the Kaloujnine-Lazard-Mal'cev correspondence (see \cite[Theorems 10.13,10.20]{Khukhro}).

Our notation is additive when inside elementary abelian sections.

		%
		%
\subsection{The functor $\Bi$}\label{sec:bi}

Let $P$ be a $p$-group of class 2 and exponent $p$, $V:=P/P'$, and 
$W:=P'$.  Then $V$ and $W$ are elementary abelian $p$-groups, that is, 
$\mathbb{Z}_p$-vector spaces.  The commutator affords an alternating 
$\mathbb{Z}_p$-bilinear map $\Bi(P):V\times V\to W$ where $b:=\Bi(P)$
is defined by
\begin{equation}\label{eq:bi-p}
	b(P'x,P'y):=[x,y],\qquad \forall x,y\in P.
\end{equation}
The radical of $b$ is $Z(P)/P'$.  If $\alpha:P\to Q$ is a
homomorphism of $p$-groups of class $2$ and exponent $p$, then 
\begin{equation}
	\Bi(\alpha):=(\alpha|_{P/P'}:P'x\mapsto Q'x\alpha,\alpha|_{P'}:x\mapsto x\alpha)
\end{equation}
is a morphism from $\Bi(P)$ to $\Bi(Q)$ (cf. \eqref{eq:pseudo-def}).  

\begin{remark}
We have refrained from using $V:=P/Z(P)$ and $W:=Z(P)$.   A homomorphism
$\alpha:P\to Q$ of $p$-groups need not map the center of $P$ into the center of $Q$.
Hence, with $W=Z(P)$ we cannot induce a morphism $\Bi(\alpha)$ of $\Bi(P)\to\Bi(Q)$.
Moreover, using $P'$ we have $W=b(V,V)$.  The penalty is that $b$ may be 
degenerate.  We avoid this difficulty by means of \lemref{lem:contract-1}.$(i)$.
\end{remark}

Given another homomorphism $\beta:Q\to R$ then $\Bi(\alpha\beta)=\Bi(\alpha)\Bi(\beta)$; so, $\Bi$ is a functor.  Finally, if $\alpha,\beta:P\to Q$ are homomorphisms then $\Bi(\alpha)=\Bi(\beta)$ if, and only if, $\alpha|_{P/P'}=\beta|_{P/P'}$ (which forces also $\alpha|_{P'}=\beta|_{P'}$).

Finally, subgroups $Q\leq P$ are mapped to $b_{QP'/P'}$ (see \eqref{eq:b-X}).  If 
$Q'=Z(Q)$ (as in \lemref{lem:contract-1}.$(i)$) then 
$Q'\leq Q\intersect P'\leq Q\intersect Z(P)\leq Z(Q)=Q'$ so that 
$Q\intersect P'=Q'$.  Hence, $QP'/P' \cong Q/Q'$ and $b_{QP'/P'}$ is naturally
pseudo-isometric to $\Bi(Q)$.  

\begin{prop}\label{prop:central-perp}
If $\mathcal{H}$ is a central decomposition of $P$, then 
$\Bi(\mathcal{H}):=\{HP'/P' : H\in\mathcal{H}\}$ is a $\perp$-decomposition 
of $b$.  
\end{prop}
\begin{proof}
Let $H$ and $K$ be distinct members of $\mathcal{H}$.  As $[H,K]=1$ it
follows that $b(HP'/P',KP'/P')=0$.  Furthermore, $\mathcal{H}$ generates
$P$ and so $\mathcal{X}:=\Bi(\mathcal{H})$ generates $V=P/P'$.  Take 
a proper subset $\mathcal{Y}\subset\mathcal{X}$.  Define 
$\mathcal{J}:=\{ H\in \mathcal{H} : HP'/P'\in\mathcal{Y}\}\subseteq\mathcal{H}$.
Note $\mathcal{Y}=\Bi(\mathcal{J})$.  Since $\mathcal{Y}$ is a proper 
subset of $\mathcal{X}$, it follows that $\mathcal{J}$ generates a proper
subgroup $Q$ of $P$ and thus $\mathcal{Y}$ generates $QP'/P'$.  We must
show $QP'/P'\neq P/P'$, or rather, that $QP'\neq P$.

Suppose that $QP'=P$.   For each $K\in \mathcal{H}-\mathcal{J}$, $K$ is not 
contained in $Q$ by the assumptions on $\mathcal{H}$.  
Now $[P:P']=[Q:Q\intersect P']\leq [QK:Q\intersect P']\leq [P:P']$ so
$QK=Q$ and $K\leq Q$.  This is impossible.  Hence $Q$ is proper.
\end{proof}

		%
		%
\subsection{The functor $\Grp$}\label{sec:grp}

Suppose $b:V\times V\to W$ is an alternating 
$\mathbb{Z}_{p}$-bilinear map.  Equip the set $V\times W$ with the product
\begin{equation*}\label{eq:inv-BCH-2}
	(u,w)*(v,x):=\left(u+v,w+x+\frac{1}{2}b(u,v)\right),\qquad \forall (u,w), 
	(v,x)\in V\times W.
\end{equation*}
The result is a group denoted $\Grp(b)$.  If $(\alpha,\hat{\alpha})$ is a morphism
from $b$ to $b':V'\times V'\to W'$ (see \eqref{eq:pseudo-def}), then 
$\Grp(\alpha,\hat{\alpha}):\Grp(b)\to \Grp(b')$
is $(v,w)\mapsto(v\alpha,w\hat{\alpha})$. 
 
By direct computation we verify that $\Grp(b)$ is a $p$-group of class $2$ and exponent
$p$ with center $\rad b\times W$ and commutator subgroup $0\times W$.  Furthermore, 
$\Grp$ is a functor.   Compare with \cite[Theorem 5.14]{Warfield:nil} and 
\cite[Theorem 2.1]{Baer:class-2}.

If $\varphi\in \Aut_{\gamma} P$ (cf. \defref{def:central-auto}) then $\varphi$ induces
the identity on $V=P/P'$ and $W=P'$.  So write $\varphi-1$ for the induced $\mathbb{Z}_p$-linear map $V\to W$ defined by $P'x(\varphi-1)=x^{-1}(x\varphi)$. 

\begin{prop}\label{prop:auto-isom}
Let $P=\Grp(b)$.  All the following hold:
\begin{enumerate}[(i)]
\item $\Aut_{\gamma} P\cong \hom(V,W)$ via the isomorphism
$\varphi\mapsto \varphi-1$, for all $\varphi\in\Aut_\gamma P$.
\item $\Aut P\cong \Isom^*(b)\ltimes \Aut_{\gamma} P$,
with $(1+\varphi)^{(\alpha,\hat{\alpha})}=1+\alpha^{-1} \varphi\hat{\alpha}$ for
each $\varphi\in\hom(V,W)$ and $(\alpha,\hat{\alpha})\in\Isom^*(b)$.
\item $C_{\Aut P}(P')\cong \Isom(b)\ltimes \Aut_{\gamma} P$.
\end{enumerate}
\end{prop}
\begin{proof}
These follow directly from the definition of $\Grp(b)$.
\end{proof}

If $U\leq V$ then define $\Grp(b_U)$ as $U\times b(U,U)\leq \Grp(b)$.  It is evident
that this determines a subgroup.  Similarly, given a set of subspaces $\mathcal{X}$
of $V$ define $\Grp(\mathcal{X})=\{\Grp(b_U) : U\in\mathcal{X}\}$.
\begin{prop}\label{prop:perp-central}
If $\mathcal{X}$ is a $\perp$-decomposition of $b$ then 
$\Grp(\mathcal{X})$ is a central decomposition of $\Grp(b)$.
\end{prop}
\begin{proof}
Let $X$ and $Y$ be distinct members of $\mathcal{X}$.  Set $H:=\Grp(b_X)$,
$K:=\Grp(b_Y)$ and $P=\Grp(b)$.  Since $b(X,Y)=0$ it follows that $[H,K]=1$.
Also, $V$ is generated by $\mathcal{X}$, and $V\times 0$ generates $P$, so that $P$ is generated by $\mathcal{H}:=\Grp(\mathcal{X})$.

Let $\mathcal{J}$ be a proper subset of $\mathcal{H}$.  Define 
$\mathcal{Y}=\{X\in\mathcal{X} : \Grp(b_X)\in\mathcal{J}\}$.  As 
$\mathcal{J}\neq\mathcal{H}$ it follows that $\mathcal{X}\neq\mathcal{Y}$ 
and therefore $U:=\langle \mathcal{Y}\rangle\neq V$.  Furthermore,
$\langle \mathcal{J}\rangle=\Grp(b_U)=U\times b(U,U)\neq V\times b(V,V)=P$.
So indeed, $\mathcal{H}$ is a central decomposition.
\end{proof}

		%
		%
\subsection{Equivalence of central and orthogonal decompositions}

Here we relate fully refined central decompositions to fully refined $\perp$-decompositions.

\begin{prop}\label{prop:semi-equiv-cat}
Let $b:V\times V\to W$ be an alternating $\mathbb{Z}_p$-bilinear map 
and $P$ a $p$-group of class 2 and exponent $p$.
\begin{enumerate}[(i)]
\item There is a natural pseudo-isometry $(\tau,\hat{\tau})$ from $b$ to 
$b':=\Bi(\Grp(b))$.
\item Every function $\ell:P/P'\to P$ to a transversal of $P/P'$ in $P$, with $0\ell=1$ determines an isomorphism 
$\varphi_{\ell}:P\to \tilde{P}$ where $\tilde{P}:=\Grp (\Bi (P))$.  
\end{enumerate}
\end{prop}
\begin{proof}
$(i)$. Let $b:V\times V\to W$ be an alternating bilinear map.  Set $P=\Grp(b)$ 
and $b'=\Bi(\Grp(b))$. Recall $P'=0\times W$ and define $\tau:V\to P/P'$ by
$v\tau=(v,0)+0\times W$ and $\hat{\tau}:W\to 0\times W$ by
$w\hat{\tau}=(0,w)$.  This makes $(\tau,\hat{\tau})$ 
a pseudo-isometry from $b$ to $b'$.  It is straightforward to verify that 
$(\tau,\hat{\tau})$ is indeed a natural transformation.  


$(ii)$. Now let $P$ be an arbitrary $p$-group of class $2$ and exponent $p$.
Set $V:=P/P'$, $W:=P'$, $b:=\Bi(P)$ and $\tilde{P}:=\Grp(\Bi(P))$.  Given 
a lift $\ell:V\to P$ with $0\ell=1$, define $x\varphi_{\ell}:=(\bar{x},x-\bar{x}\ell)$
where $\bar{x}:=P'x$.   The group $P$ has the presentation 
	\[\langle V\ell,W \mid [u\ell,v\ell]=b(u,v),
		\textnormal{ exponent }p,\textnormal{ class }2\rangle\]
and $\tilde{P}$ has the presentation 
	\[\langle V\times 0, 0\times W \mid [(u,0),(v,0)]=(0,b(u,v)),
		\textnormal{ exponent }p, \textnormal{ class }2\rangle.\]  
Evidently $\varphi_{\ell}$ preserves the exponent relations.  Furthermore,
\begin{equation*}
	[x,y]\varphi_{\ell}=[\bar{x}\ell,\bar{y}\ell]\varphi_{\ell}
		= b(\bar{x},\bar{y})\varphi_{\ell} = ( 0 , b(\bar{x},\bar{y}))
\end{equation*}
for each $x,y \in P$.  Hence, $\varphi_{\ell}$ preserves all the relations of the
presentations and so $\varphi_{\ell}$ is a homomorphism, indeed, an isomorphism.

\end{proof}

\begin{thm}\label{thm:main-reduction}
Let $P$ be a $p$-group of class $2$ and exponent $p$ such that $P'=Z(P)$, and let
$\mathcal{H}$ be a central decomposition of $P$.
\begin{enumerate}[(i)]
\item $P$ is centrally indecomposable if, and only if, $\Bi(P)$ is $\perp$-indecomposable.
\item $\mathcal{H}$ is a fully refined if, and only if, $\Bi(\mathcal{H})$ 
is fully refined.
\item if $\mathcal{K}$ is a central decomposition of $P$, then
	\begin{enumerate}[(a)]
	\item there is an automorphism $\alpha\in \Aut P$ such that 
		$(\mathcal{H}P')\alpha=\mathcal{K}P'$ if, and only if, there is a 
		$(\beta,\hat{\beta})\in\Isom^*(\Bi(P))$ such that 
		$(\Bi(\mathcal{H}))\beta=\Bi(\mathcal{K})$.
	\item there is an automorphism $\alpha\in C_{\Aut P}(P')$ such that 
		$(\mathcal{H}P')\alpha=\mathcal{K}P'$ if, and only if, there is a 
		$\beta\in\Isom(\Bi(P))$ such that 
		$(\Bi(\mathcal{H}))\beta=\Bi(\mathcal{K})$.
	\end{enumerate}
\end{enumerate}
\end{thm}
\begin{proof}
$(i)$.  Let $P$ be a centrally indecomposable group and take $b:=\Bi(P)$,
$V=P/P'$, $W=P'$.  Suppose that $\mathcal{X}$ is a $\perp$-decomposition of $b$.  It follows that $\{X\times b(X,X):X\in\mathcal{X}\}$ is central decomposition of $\Grp(\Bi(P))$, \propref{prop:perp-central}.  By \propref{prop:semi-equiv-cat}.$(ii)$ we know $P$ is isomorphic to $\Grp(\Bi(P))$ so that $\Grp(\Bi(P))$ must be centrally indecomposable.  Therefore, $X\times b(X,X)=\Grp(\Bi(P))=V\times W$ so that $X=V$, for each $X\in\mathcal{X}$.  Since no proper subset of $\mathcal{X}$ generates $V$ it follows that $\mathcal{X}=\{V\}$ and $b$ is $\perp$-indecomposable.

Next suppose that $b$ is $\perp$-indecomposable and that $P=\Grp(b)$.
Suppose that $\mathcal{H}$ is a fully refined central decomposition of $P$.  Then
$\{HP'/P' : H\in\mathcal{H}\}$ is a $\perp$-decomposition of $\Bi(\Grp(b))$, \propref{prop:central-perp}.  \propref{prop:semi-equiv-cat}.$(i)$ states that $b$ is pseudo-isometric
to $\Bi(\Grp(b))$ and so $HP'/P'=P/P'$, or rather $HP'=P$, for each $H\in\mathcal{H}$.
Hence $H'=P'$ for each $H\in\mathcal{H}$.  Since $P'\neq 1$ there is an $H\in \mathcal{H}$
which is non-abelian.  Furthermore, $H$ is centrally indecomposable so that by
\lemref{lem:contract-1}.$(i)$, $H'=Z(H)$.  Therefore, $HP'=H\oplus A$ for 
some $A\leq Z(P)$ such that $H'A=P'$, \lemref{lem:contract-1}.$(i)$.  But 
$H'=P'$ forces $A=1$.  Thus $H=P$, and $P$ is centrally indecomposable.

$(ii)$.  This follows from \propref{prop:perp-central}, \propref{prop:central-perp} and $(i)$.  Finally, $(iii)$ follows from \propref{prop:auto-isom}.
\end{proof}

\begin{ex}\label{ex:cent-perp}
If $H$ is $p$-group of class $2$ and exponent $p$ with $b=\Bi(H)$ then
\begin{equation*}
	\Bi(\overbrace{H\circ\cdots\circ H}^n) = \overbrace{b\perp\cdots \perp b}^n,
\end{equation*}
(cf. \eqref{eq:canonical-cent}).
Furthermore, the canonical central decomposition $\{H_1,\dots, H_n\}$ of 
$\overbrace{H\circ\cdots\circ H}^n$ corresponds to the canonical $\perp$-decomposition
$\{V_1,\dots, V_n\}$ of $\overbrace{b\perp\cdots\perp b}^n$.
\end{ex}

%
%
\section{Adjoint and self-adjoint operators}\label{sec:adj-sym}

This section proves a structure theorem for isometry groups (\thmref{thm:isom-classical}), introduces a criterion for groups/bilinear maps to be indecomposable (\thmref{thm:indecomps}), and proves a stronger version of the first part of \thmref{thm:central-count} (\thmref{thm:central-III}).

Throughout this section let $b:V\times V\to W$ be a \emph{non-degenerate 
Hermitian bilinear map over a field $k$} (cf. \eqref{eq:herm}).  We associate to 
$b$ a $*$-algebra, and a Hermitian Jordan algebra of \emph{self-adjoint} elements.  The isometry group of $b$ is a subgroup of the group of units of the $*$-algebra and $\perp$-decompositions are represented by sets of pairwise orthogonal idempotents of the Jordan algebra. 

		%
		%
\subsection{The adjoint $*$-algebra $\Adj(b)$}

\begin{defn}\label{def:*-alge}
\begin{enumerate}[(i)]
\item A map $f\in\End V$ has an \emph{adjoint} $f^*\in \End V$ for $b$ if  
	\[b(uf,v)=b(u,vf^*),\qquad \forall u,v\in V.\]  
Write $\Adj(b)$ for the set of all endomorphisms with an adjoint for $b$.
\item A $*$-algebra is an associative $k$-algebra $A$ with a linear bijection 
$*:A\to A$ such that $(ab)^*=b^* a^*$ and $(a^*)^*=a$ for all $a,b\in A$. 
\item A homomorphism $f:A\to B$ of $*$-algebras is a $*$-homomorphism if
$a^*f=(af)^*$ for all $a\in A$.
\item The \emph{trace} of $A$ is $T(x)=x+x^*$ for all $x\in A$.
\item The \emph{norm} of $A$ is $N(x)=xx^*$ for all $x\in A$.
\end{enumerate}
\end{defn}

\begin{prop}\label{prop:adjoint}
$\Adj(b)$ is an associative unital $*$-algebra; in particular, adjoints are unique.
\end{prop}
\begin{proof}
Let $f\in \Adj(b)$ and $f',f''\in\End V$ where $b(u,vf')=b(uf,v)=b(u,vf'')$
for all $u,v\in V$.  As $b$ is non-degenerate, $vf'=vf''$ so that $f'=f''$.  
If $f,g\in\Adj(b)$ then $b(u(fg),v)=b(uf,vg^*)=b(u,v(g^* f^*))$ 
for $u,v\in V$; so, $fg\in \Adj(b)$ with $(fg)^*=g^* f^*$.
Since $b(u,v)=b(v,u)\theta$ for $u,v\in V $, it follows
that $b(uf^*,v)=b(v,uf^*)\theta=b(vf,u)\theta=b(u,vf)$ for every $f\in \Adj(b)$.
Hence, $f^*\in\Adj(b)$ and $(f^*)^*=f$.  
\end{proof}

\begin{prop}\label{prop:groups}
Let $b:V\times V\to W$ and $b':V'\times V'\to W'$ be non-degenerate Hermitian maps.
\begin{enumerate}[(i)]
\item A pseudo-isometry $(\alpha,\beta)$ from $b$ to $b'$ (cf. \eqref{eq:pseudo-def}) induces a $*$-isomorphism 
	\[f\mapsto f^{(\alpha,\beta)}:=\alpha^{-1} f\alpha\] 
of $\Adj(b)$ to $\Adj(b')$.  In particular, $\Isom^*(b)$ acts on $\Adj(b)$.
\item Let $\varphi\in \GL(V)$ and $s\in k^\times$.  Then
$(\varphi,s 1_W)\in\Isom^*(b)$ if, and only if,
$\varphi\in\Adj(b)$ and $\varphi\varphi^*=s 1_V$.  Hence,
	\[\Isom(b) =\{\varphi\in \Adj(b):\varphi \varphi^*=1_V\}.\]
\end{enumerate}
\end{prop}
\begin{proof}
$(i)$ We have
\begin{align*}
b'(u f^{(\alpha,\beta)},v)
	& = b'(u\alpha^{-1} f \alpha,v\alpha^{-1}\alpha)
	=b(u\alpha^{-1} f,v\alpha^{-1})\beta\\
	& =b(u\alpha^{-1},v\alpha^{-1}f^*)\beta
	=b'(u,v(f^*)^{(\alpha,\beta)}),
\end{align*}
for each $u,v\in V'$ and $f\in \Adj(b)$.  Hence $f^{(\alpha,\beta)}\in \Adj(b')$ with 
$(f^{(\alpha,\beta)})^*=(f^*)^{(\alpha,\beta)}$.  

$(ii)$ Take $(\varphi,s 1_W)\in\Isom^*(b)$, $s\in k^\times$. Then
\[
	b(u\varphi,v)=b(u\varphi,v\varphi^{-1}\varphi)
		=s b(u,v\varphi^{-1})=b(u,sv\varphi^{-1}), \qquad \forall u,v\in V.\]
Hence $\varphi\in\Adj(b)$ with $\varphi^*=s\varphi^{-1}$.
Conversely, if $\varphi\in\Adj(b)$ with $\varphi\varphi^*=s 1_V$ then 
	\[b(u\varphi,v\varphi)=b(u,v\varphi\varphi^*)=s b(u,v),\qquad \forall u,v\in V.\]
Thus $(\varphi,s 1_W)\in \Isom^*(b)$.
\end{proof}

\subsection{Simple $*$-algebras and Hermitian $C$-forms $d:V\times V\to C$}\label{sec:matrix}

In this section we summarize in a uniform manner the known results of finite simple $*$-algebras (\thmref{thm:simple-*-alge}) and the corresponding finite classical groups (\propref{prop:isom-type}).
\begin{thm}\label{thm:simple-*-alge}
A finite simple $*$-algebra of odd characteristic is $*$-isomorphic to one of the following 
for some $n\in\mathbb{N}$ and some field $K$:
\begin{description}
\item[Orthogonal case] $M_n(K)$ with the $X\mapsto DX^tD^{-1}$ as the involution, 
for $X\in M_n(K)$, where $D$ is either $I_n$ or $I_{n-1}\oplus [\omega]$ and $\omega\in K$ is a non-square (compare \eqref{eq:dot-prod}).
\item[Unitary case] $M_n(F)$ with involution $X\mapsto \bar{X}^t$, where 
$F/K$ is a quadratic field extension with involutory field
automorphism $x\mapsto \bar{x}$, $x\in F$, applied to the entries of $X\in M_n(F)$.
\item[Exchange case] $M_n(K\oplus K)$ with involution $X\mapsto \bar{X}^t$, where
$\overline{(x,y)}:=(y,x)$ for $(x,y)\in K\oplus K$, defines an involution on 
$K\oplus K$ which is applied to the entries of $X\in M_n(K\oplus K)$,
\item[Symplectic case] $M_n(M_2(K))$ with involution $X\mapsto \bar{X}^t$, where
\begin{equation}\label{eq:adjugate} 
\overline{\begin{bmatrix} a & b\\ c & d\end{bmatrix}}
	:=\begin{bmatrix} d & -b\\ -c & a\end{bmatrix}
	=\begin{bmatrix} 0 & 1\\ -1 & 0 \end{bmatrix}
	\begin{bmatrix} a & b \\ c & d\end{bmatrix}
	\begin{bmatrix} 0 & 1 \\ -1 & 0 \end{bmatrix}^{-1}
\end{equation}
defines an involution on $M_2(K)$ which is applied to each entry of $X\in M_n(M_2(K))$.
\end{description}
\end{thm}
\begin{proof} See \cite[p.178]{Jacobson:Jordan} restricting consideration to finite fields.
(Compare with \thmref{thm:Hurwitz}, \propref{prop:orthogonal}, \eqref{eq:dot-prod}, 
and \corref{coro:orthonormal}.)
\end{proof}

The above description of these algebras will allow us to give uniform proofs later;
however, there are simpler and more standard descriptions, for example:
\begin{remark}\label{rem:exchange-symplectic}
The exchange type $*$-algebras can also be described as $M_n(K)\oplus M_n(K)$ with $(X,Y)^*=(Y^t, X^t)$ for $(X,Y)\in M_n(K)\oplus M_n(K)$. 

The symplectic type $*$-algebras are $*$-isomorphic to $M_{2n}(K)$ with involution $X^*=J X^t J^{-1}$, for each $X\in M_{2n}(K)$, where $J:=I_n\otimes \begin{bmatrix} 0 & 1\\ -1 & 0\end{bmatrix}$ \cite[p. 178]{Jacobson:Jordan}.
\end{remark}

\begin{defn}\label{def:comp}\cite[Definition 6.2.2]{Jacobson:Ark}
A $*$-algebra $C$ is an associative \emph{composition} algebra over a field $K$ 
(where by convention $x^*$ is denoted $\bar{x}$ and \emph{bar} replaces $*$ 
in the definitions) if 
\begin{enumerate}[(i)]
\item $K=\{x\in C: x=\bar{x}\}$ and
\item $xax=0$ for all $a\in C$ implies $x=0$.
\end{enumerate}
\end{defn}

\begin{thm}\cite[Theorem 6.2.3]{Jacobson:Ark}\label{thm:Hurwitz}
Over a finite field $K$ of odd characteristic each associative composition algebra $C$ 
is bar-isomorphic to one of the following:
\begin{enumerate}[(i)]
\item $K$ with trivial involution,
\item a quadratic field extension $F/K$ with the involutorial field automorphism, 
\item $K\oplus K$ with the \emph{exchange} involution $\overline{(x,y)}=(y,x)$ for $(x,y)\in K\oplus K$, or
\item $M_2(K)$ with the involution \eqref{eq:adjugate}.
\end{enumerate}
In particular these algebras are simple bar-algebras and with the exception of exchange also simple algebras.  Norms (cf. \defref{def:*-alge}.$(v)$) behave as follows: $N(C)=K$ if $C>K$; otherwise, $N(K)=K^2$. 
\end{thm}

\begin{defn}\label{def:Hermitian}
Let $C$ be an associative composition algebra and $V$ be a free left $C$-module.  
We call a
$K$-bilinear map $d:V\times V\to C$ a \emph{Hermitian $C$-form} if, for $u,v\in V$ 
and $s\in C$, it follows that:
\begin{enumerate}[(i)]
\item $d(u,v)=\overline{d(v,u)}$, and
\item $d(su,v)=sd(u,v)$ and $d(u,sv)=d(u,v)\bar{s}$.
\end{enumerate}
The \emph{rank} of $d$ is the rank of $V$ as a free left $C$-module.
\end{defn}
Note that a Hermitian $C$-form is also a Hermitian $K$-bilinear map and the usual
definitions of (pseudo-)isometries apply.  \emph{It is most important to note that
$d(x,x)=\overline{d(x,x)}$; hence, $d(x,x) \in K$, for all $x\in V$}.

Let $C$ be an associative composition algebra over $K$ and $D\in M_n(C)$ where 
$D=\bar{D}^t$.  Then $d_D(u,v):=uD\bar{v}^t$, for $u,v\in C^n$, determines 
a Hermitian $C$-form $d_D:C^n\times C^n\to C$.  Here adjoints $f,f^*\in\Adj(d_D)$ can be represented as matrices $F,F^*\in M_n(C)$ such that:
\begin{equation*}
	uFD\bar{v}^t=d_D(uf,v)=d_D(u,vf^*)=uD(\bar{F}^*)^t \bar{v}^t,
		\qquad \forall u,v\in C^n.
\end{equation*}
Hence, $FD=D(\bar{F}^*)^t$.  As $D$ is invertible, $\Adj(d_D)$ $*$-isomorphic
to $M_n(C)$ with involution defined by
\begin{equation}\label{def:*}
	F^*  := D \bar{F}^t D^{-1},\qquad\forall F\in M_n(C).
\end{equation}

Likewise, if $d:V\times V\to C$ is a Hermitian $C$-form
and $\mathcal{X}$ is an ordered basis of $V$ as a free left $C$-module, then setting 
$D_{xy}:=d(x,y)$, for all $x,y\in\mathcal{X}$, determines a matrix $D$ in $M_n(C)$, 
$n=|\mathcal{X}|$, such that $D=\bar{D}^t$ and the Hermitian $C$-form given
by $D$ is isometric to $d$.  Furthermore, $d$ is non-degenerate if, and only if, $D$ is
invertible.  So we have:
\begin{coro}\label{coro:*-simple}
Every simple $*$-algebra is $*$-isomorphic to $\Adj(d)$ for a non-de\-gen\-er\-ate
Hermitian $C$-form $d:V\times V\to C$.
\end{coro}

In the cases where $C$ has orthogonal or unitary type we have the usual symmetric and Hermitian forms, respectively.  Suppose instead the $C=M_2(K)$ and that 
$d:V\times V\to C$ is the non-degenerate Hermitian $C$-from $d(u,v):=u\bar{v}^t$,
where  $V=C^n$.  There is a natural submodule $U$ of $V$ defined by:
\begin{align*}
	U:=
	\overbrace{
	\left\{\begin{bmatrix} * & * \\ 0 & 0 \end{bmatrix}\right\}\oplus\cdots\oplus
	\left\{\begin{bmatrix} * & * \\ 0 & 0 \end{bmatrix}\right\}}^n\leq V.
\end{align*}
Furthermore, $d(U,U)\cong K$; hence, the restriction $d_U:U\times U\to K$ is a bilinear form.  It is easily checked that $d_U$ is alternating and non-degenerate.  The case when $C$ has exchange type is not usually handled as a form but for a uniform treatment we find it convenient.  In particular we may state: 

\begin{defn}\label{def:non-singular}
Given a non-degenerate Hermitian $C$-form $d:V\times V\to C$, an element $x\in V$
is \emph{non-singular} if $d(x,x)\neq 0$ and $\dim Cx=\dim C$.
\end{defn}

\begin{prop}\label{prop:orthogonal}
Every non-degenerate Hermitian $C$-form $d:V\times V\to C$ has an orthogonal 
$C$-basis $\mathcal{X}$ (i.e.: $\mathcal{X}$ is a $C$-basis for $V$ and $d(x,y)=0$ if $x\neq y$, $x,y\in \mathcal{X}$).  Furthermore, every fully refined $\perp$-decomposition of $d$ determines an orthogonal basis and so every $\perp$-indecomposable has rank $1$.
\end{prop}
\begin{proof}
First we show that there is always a non-singular vector $x\in V$.  

Suppose otherwise: $d(x,x)=0$ for any $x\in V$ such that $\dim Cx=\dim C$.  Immediately,
$d(v,v)=0$ for all $v\in V$ and thus $-d(v,u)=d(u,v)=\overline{d(v,u)}$ for $u,v\in V$.  

For each $u\in V$, $Cd(u,V)+d(V,u)C$ is a bar-ideal of $C$.
As $C$ is a simple bar-algebra (\thmref{thm:Hurwitz}), $Cd(u,V)+d(V,u)C=0$ or $C$.
If $Cd(u,V)+d(V,u)C=0$ then $Cd(u,V)=0$ and $d(V,u)C=0$; hence, $u\in\rad d=0$.  Thus, $C=Cd(u,V)+d(V,u)C$ for all $u\in V-\{0\}$.  We split into two cases.

If $C=Cd(u,V)$ then $1=d(su,v)$ for some $s\in C$ and $v\in V$.  Then
$1=\bar{1}=\overline{d(su,v)}=-d(su,v)=-1$, so $\chr K=2$, which we exclude.  
Similarly, $d(V,u)C\neq C$. 

Now suppose $C\neq Cd(u,V), d(V,u)C$.  Then $Cd(u,V)$ is a proper ideal of $C$.
Form \thmref{thm:Hurwitz} this requires $C=K\oplus K$ with the exchange involution.
Without loss of generality, take $Cd(u,V)=K\oplus 0$.  Hence $(1,0)=sd(u,v)$ 
for some $s\in C$ and $v\in V$.  Thus,
$(1,1)=d(su,v)+\overline{d(su,v)}=d(su,v)-d(su,v)=0$, which is a lie.
Therefore, there exists a non-singular vector $x\in V$.  

As $0\neq d(x,x)=\overline{d(x,x)}$ it follows that $d(x,x)\in K^\times$.
Then 
$d\left(v-\frac{d(v,x)}{d(x,x)}x,x\right)
=d(v,x)-\frac{d(v,x)}{d(x,x)}d(x,x)=0$, for $v\in V$.
That is, $v-\frac{d(v,x)}{d(x,x)}x\in x^\perp$; hence, 
$v=\frac{d(v,x)}{d(x,x)}x+\left(v-\frac{d(v,x)}{d(x,x)}x\right)$ shows that 
$V=Cx+x^\perp$.  Since
$Cx\intersect x^\perp=0$ it follows that $V=Cx\oplus x^\perp$.  
Restrict $d$ to $x^\perp$ and induct to 
exhibit an orthogonal basis $\mathcal{X}$ for $d$ on $x^\perp$.  Thus
$\mathcal{X}\union\{x\}$ is an orthogonal basis of $d$ on $V$.
\end{proof}

Notice in the case of symplectic type, if $\{x_1,\dots, x_n\}$ is an orthogonal $C$-basis for $d$, then $V=Cx_1\perp\cdots\perp Cx_n$.  Translating to the associated alternating bilinear form $d'$, the orthogonal basis becomes a hyperbolic basis:
$U=H_1\perp\cdots\perp H_n$ where each $H_i$ is a hyperbolic line (cf. 
\cite[Definition 3.5]{Artin:geometry}).  In the case of exchange type, a natural orthogonal
basis is given by $\{(x,x):x\in \mathcal{X}\}$ where $\mathcal{X}$ is a $K$-basis of 
$U$ and $V=U\oplus U$, $U=K^n$.

\begin{coro}\label{coro:orthonormal}
If $C$ does not have orthogonal type then $d$ has an orthonormal $C$-basis (i.e.: a basis $\mathcal{X}$ where $d(x,y)=\delta_{xy}$, for all $x,y\in\mathcal{X}$).
In particular, $d$ is pseudo-isometric to the \emph{$C$-dot product} 
$d:C^n\times C^n\to C$ where $d(u,v):=u\bar{v}^t$, for all $u,v\in V$.
\end{coro}
\begin{proof}
From \thmref{thm:Hurwitz}, $N(C)=K$ whenever $C>K$.  Therefore if $v\in V$
such that $d(v,v)\neq 0$ then $d(v,v)=N(s)=s\bar{s}$ for some $s\in C^\times$.
Let $u=s^{-1} v$ so that $d(u,u)=s^{-1}d(v,v)\bar{s}^{-1}=s^{-1}N(s)\bar{s}^{-1}=1$.  By \propref{prop:orthogonal}, we have an orthogonal basis $\mathcal{X}$ for $d$.  Replace each $x\in\mathcal{X}$ with $s_x^{-1} x$ so that $d(s^{-1}_x x,s^{-1}_x x)=1$ and $\{s^{-1}x:x\in\mathcal{X}\}$ is still an orthogonal $C$-basis.
\end{proof}

\begin{prop}\label{prop:isom-type}
Let $d:V\times V\to C$ be a non-degenerate Hermitian $C$-form.  Then
$\Adj(d)=\End V\cong M_n(C)$ as an algebra, and the following hold:
\begin{description}
\item[Orthogonal case] $C=K$ and $\Isom(d)=\GO(d)$;
\item[Unitary case] $C=F$ and $\Isom(d)=\GU(d)$;
\item[Exchange case]  $C=K\oplus K$, $\Isom(d)\cong\GL(U)$, $V=U\oplus U$; and
\item[Symplectic case] $C=M_2(K)$ and $\Isom(d)\cong\Sp(U)$, $V=U\oplus U$.
\end{description}
\end{prop}
\begin{proof}
The first two cases are by definition alone.
If $C=K\oplus K$ then $\Adj_C(d)\cong \End U\oplus \End U$ with 
$(f\oplus g)^*=g\oplus f$.  Hence, the isometry group is:
\begin{equation*}
\begin{split}
\Isom(d) & =\{f\oplus g\in \GL(U)\oplus \GL(U): (f\oplus g)(f\oplus g)^*=1\oplus 1\}\\
	& =\{f\oplus f^{-1}: f\in \GL(U)\}\cong \GL(U).
\end{split}
\end{equation*}
Finally, if $C=M_2(K)$ then $\Adj(d)\cong \Adj(d')$ where $d'$ is the non-degenerate alternating $K$-bilinear form on $U$, \remref{rem:exchange-symplectic}.  Therefore $\Isom(d)\cong\Isom(d')$ as both are the set of elements defined by $\varphi\varphi^*=1$ (\propref{prop:groups}.$(ii)$).  The latter group is by definition $\Sp(U)$.
\end{proof}

		%
		%
\subsection{Radical and semi-simple structure of $*$-algebras}

\begin{defn}
\begin{enumerate}[(i)]
\item A \emph{$*$-ideal} is an ideal $I$ of a $*$-algebra $A$ such that $I^*= I$.
\item $\spec_0 A$ is the set of all maximal $*$-ideals of $A$.
\item A \emph{$*$-simple} algebra is a $*$-algebra with exactly two $*$-ideals.  
\item A \emph{$*$-semi-simple} algebra is a direct product of simple $*$-algebras.
\item A $*$-ideal is \emph{nil} if it consists of nilpotent elements.
\end{enumerate}
\end{defn}

\begin{thm}[$*$-algebra structure theorem]\label{thm:*-structure}
Let $A$ be a $*$-algebra with Jacobson radical $\rad A$.  Then
\begin{enumerate}[(i)]
\item $\rad A$ is a nil $*$-ideal,
\item $A/\rad A$ is $*$-semi-simple, and
\item if $A$ is $*$-simple then $A\cong \Adj(d)$ for a non-degenerate Hermitian $C$-form $d$.
\end{enumerate}
\end{thm}
\begin{proof}
$(i)$ Since $*$ is an anti-automorphism of $A$, every left quasi-regular element is mapped
to a right quasi-regular element.  Thus $(\rad A)^*\subseteq \rad A$.  Since $A$ is
finite dimensional, the Jacobson radical is nilpotent.

$(ii)$ We induce $*$ on $A/\rad A$, so that $A/\rad A$ is a $*$-algebra
which is product of uniquely determined minimal ideals. If $I$ is a minimal ideal of 
$A/\rad A$ then either $I^*=I$ or $I\intersect I^*=0$ so that 
$\langle I,I^*\rangle=I\oplus I^*$ is a minimal $*$-closed ideal.  Thus 
$A/\rad A$ is a product of simple $*$-algebras.

For $(iii)$ see Section \ref{sec:matrix}.
\end{proof}

		%
		%
\subsection{Isometry groups are unipotent-by-classical}

We describe the structure of the isometry group of a Hermitian bilinear map.  
To do this we invoke the following generalization 
of the Wedderburn Principal Theorem for finite dimensional $*$-algebras over fields not of characteristic 2 (cf. \cite{Lewis:*-survey}).
\begin{thm}\cite[Theorem 1]{Taft:Wedderburn}\label{thm:Taft}
Given a finite dimensional $*$-algebra $A$ over a separable field $k$, there is
a subalagebra $B$ of $A$ such that $B^*=B$, $A=B\oplus \rad A$ as a $k$-vector space,
and $B\cong A/\rad A$.
\end{thm}

Recall that the $p$-\emph{core} of a finite group $G$, denoted $O_p(G)$, 
is the largest normal $p$-subgroup of $G$.
\begin{thm}\label{thm:isom-classical}
If  $\Adj(b)/\rad \Adj(b)\cong \Adj(d_1)\oplus\cdots\oplus \Adj(d_s)$
where $d_i$ is a non-degenerate Hermitian $C_i$-form, for some
associative composition algebra $C_i$, for each $1\leq i\leq s$, then 
\begin{equation*}
\Isom(b)\cong (\Isom(d_1)\times\cdots\times \Isom(d_s))\ltimes
	O_p(\Isom(b)),
\end{equation*}
where $p$ is the characteristic of $\Adj(b)$.
\end{thm}
\begin{proof}
Let $A:=\Adj(b)$.  By \thmref{thm:Taft} we have $A=B\oplus \rad A$ where the
projection map $\pi:A\to B$ is a surjective $*$-homomorphism with kernel $\rad A$.
Now set $G=\{\varphi\in B:\varphi\varphi^*=1\}$ and $N=\{\varphi\in A:
\varphi\varphi^*=1, \varphi-1\in\rad A\}$.  If $\varphi=1+z,\tau=1+z'\in N$, $z,z'\in\rad A$, then $\varphi\tau-1=z+z'+zz'\in \rad A$ so that $\varphi\tau\in N$.  Hence, $G$ and $N$ are subgroups of $\Isom(b)$ and $G\intersect N=1$.  As $\pi$ is a $*$-homomorphism, $(\varphi\pi)(\varphi\pi)^*=(\varphi\varphi^*)\pi=1$ for all $\varphi\in \Isom(b)\subset A$ (\propref{prop:groups}.$(ii)$).   Hence, $\Isom(b)\pi=G$.  Finally, the kernel of $\pi$ restricted to $\Isom(b)$ is $N$.  Thus $\Isom(b)=G\ltimes N$.  Since $B\cong A/\rad A\cong \Adj(d_1)\oplus\cdots\oplus\Adj(d_s)$ it follows that $G\cong \Isom(d_1)\times\cdots\times \Isom(d_s)$ (\propref{prop:groups}.$(ii)$).  By \propref{prop:isom-type}, $O_p(G)=1$.  Thus, $O_p(\Isom(b))=N$.
\end{proof}

		%
		%
\subsection{The Jordan algebra $\Sym(b)$ of self-adjoint operators}
		\label{sec:Jordan}
		
At last we introduce the Jordan algebras associated to our bilinear maps (and 
thus to our $p$-groups as well).

\begin{defn}
For a $k$-bilinear map $b:V\times V\to W$, define
	\[\Sym(b) := \{f\in \End V: b(uf,v)=b(u,vf), \forall u,v\in V\}\]
\end{defn}
(The notation $\Sym(b)$ has no relationship to symmetric groups.)
This is an instance of a broader class of objects (see \thmref{thm:jordan}):
\begin{defn}\label{def:H(A,*)}
Given a $*$-algebra $A$, the \emph{special Hermitian Jordan algebra} of $A$ is
the set
	\[\mathfrak{H}(A,*)=\{a\in A: a=a^*\}\] 
equipped with the special Jordan 
product $x\bullet y=\frac{1}{2}(xy+yx)$ \cite[pp. 12-13]{Jacobson:Jordan}.  
\end{defn}
Special Hermitian Jordan algebras are part of the family of unital Jordan algebras, which are algebras $J$ with a binary product $\bullet$ such that:
\begin{enumerate}[(i)]
\item\label{ax:1} $x\bullet y=y\bullet x$,
\item\label{ax:2} $x^{\bullet 2}\bullet (y\bullet x)=(x^{\bullet 2}\bullet y)\bullet x$ 
where $x^{\bullet 2}=x\bullet x$, and
\item\label{ax:3} $x\bullet 1=1\bullet x=x$
\end{enumerate}
for all $x,y\in J$ \cite[Definition I.2]{Jacobson:Jordan}.  \emph{Unless stated otherwise,
our use of Jordan algebras is restricted to finite special Hermitian Jordan algebras.}  As we deal only with odd characteristic, the definitions we provide for ideals, powers, and related properties are in terms of the classical $x\bullet y$ product rather than the quadratic Jordan definitions.  This said, we still have many uses for the \emph{quadratic
Jordan product} which in a special Hermitian Jordan algebra $J:=\mathfrak{H}(A,*)$ is
simply:
\begin{equation}\label{eq:quadratic-prod}
	yU_x := xyx \in J,\qquad x,y\in J.
\end{equation}
Evidently the Jordan product $\bullet$ need not be associative.  However, 
we always have $x^i\bullet x^j=\frac{1}{2}(x^{i+j}+x^{j+i})=x^{i+j}$, 
$i,j\in\mathbb{N}$ (cf. \cite[p. 5]{Jacobson:Jordan}).  As $J=\mathfrak{H}(A,*)$ and $1^*=1$, the identity of $J$ is the identity of $A$.  Furthermore, if $x\in J$ is invertible in $A$ then $(x^{-1})^*=(x^*)^{-1}=x^{-1}$ proving that $x^{-1}\in J$.  Hence we omit the $\bullet$ notation in the exponents of our Jordan algebra products. 

From our discussion thus far we have:
\begin{thm}\label{thm:jordan}
For every non-degenerate Hermitian bilinear map $b$,
$\Sym(b)$ is the special Hermitian Jordan algebra $\mathfrak{H}(\Adj(b))$.
Furthermore, $\Isom^*(b)$ acts on $\Sym(b)$ as in \propref{prop:groups}.
\end{thm}
\begin{proof}
This follows directly from the definitions.
\end{proof}

\begin{defn}\cite[4.1-4.2]{Jacobson:Ark}
Let $J$ be a Jordan algebra.
\begin{enumerate}[(i)]
\item A subspace $I$ of $J$ is an \emph{ideal} if $I\bullet J\subseteq I$.  
Then, in the usual way, $J/I$ becomes a Jordan algebra.
\item A \emph{nil} ideal is an ideal that consists of nilpotent elements.
\item A subspace $I$ is an \emph{inner} ideal if $JU_I=\{aU_b:a\in J, b\in I\}\subseteq I$.  
\item The \emph{radical}, denoted $\rad J$, is the intersection 
of all maximal inner ideals \cite[4.4.10]{Jacobson:Ark}.
\item $J$ is \emph{simple} if it has exactly two ideals, and \emph{semi-simple} 
if it is a direct product of simple Jordan algebras.
\end{enumerate}
\end{defn}

In Jordan algebras, the inner ideals often play the r\^{o}le that left/right ideals play
for associative algebras.  Every ideal of a Jordan algebra is also an inner ideal.  As $J=\mathfrak{H}(A,*)$ (cf. \defref{def:H(A,*)}) each ideal $I$ of $A$ determines an ideal $I\intersect J$ of $J$.  Likewise, if $I$ is a left or right ideal of $A$ then $I\intersect J$ is an inner ideal.    For further details see \cite[4.1-4.2]{Jacobson:Ark}.

We can account for all the special simple Hermitian Jordan algebras (also called
\emph{special Jordan matrix algebras}) in much the same way
as we have describe the simple $*$-algebras.
\begin{defn}\cite[III.2]{Jacobson:Jordan}
Let $C$ be a finite associative composition algebra over a field $K$ and 
$D=\Diag[\omega_1,\dots, \omega_n]$ a matrix in $M_n(C)$ with
entries in $K^\times$.  Then the \emph{special Jordan matrix algebra} with respect 
to $D$ is
	\[\mathfrak{H}(D)=\{X\in M_n(C): X=D\bar{X}^tD^{-1}\}\]
whose product is $X\bullet Y=\frac{1}{2}(XY+YX)$ and where $XU_Y=YXY$ for
$X,Y\in\mathfrak{H}(D)$.
\end{defn}

Following Section \ref{sec:matrix} we know $d(u,v):=uD\bar{v}^t$, $u,v\in C^n$, 
determines a non-degenerate Hermitian $C$-form and
\begin{equation}
	\mathfrak{H}(D)=\mathfrak{H}(\Adj(d))=\Sym(d).
\end{equation}
By \cite[p.178-179]{Jacobson:Jordan}, $\mathfrak{H}(D)$ is a special simple Hermitian Jordan algebra (though typically the case of $C=K\oplus K$ is not specified in this manner).

\begin{thm}[Hermitian Jordan algebra structure theorem]\label{thm:herm-structure}
Let $A$ be a finite $*$-algebra with Jacobson radical $\rad A$, and let 
$J=\mathfrak{H}(A,*)$.
\begin{enumerate}[(i)]
\item $\rad J=J\intersect\rad A$ and is a nil ideal of $J$,
\item $J/\rad J$ is a semi-simple Jordan algebra,
\item every special simple Hermitian Jordan algebra is isomorphic to $\Sym(d)$ 
for some non-degenerate Hermitian $C$-form $d$.
\item for every $I\in\spec_0 A$, $J\intersect I$ is a maximal ideal of $J$.
\end{enumerate}
\end{thm}
\begin{proof}
$(iii)$. This follows from \cite[pp.178-179, Second Structure Theorem]{Jacobson:Jordan}.  

$(ii)$. This follows from $(iii)$ and \thmref{thm:*-structure}.$(ii)$, 
$J/(J\intersect \rad A)=\mathfrak{H}(A/\rad A,*)$ is semi-simple. 

$(i)$. By \cite[p.161, First Structure Theorem]{Jacobson:Jordan} (interpreted in radical vocabulary in \cite[4.2.7,4.2.15]{Jacobson:Ark}),  $\rad (J/\rad J)=0$ and also $\rad J=0$ if, and only if, $J$ is semi-simple.  Thus, by $(iii)$, it follows that $J\intersect\rad A=\rad J$.  By \thmref{thm:*-structure}.$(i)$, $\rad A$ is a nil ideal, and so $\rad J=J\intersect \rad A$ is also a nil ideal.

$(iv)$.  This is immediate from $(iii)$ and \thmref{thm:*-structure}.$(iii)$.
\end{proof}

		%
		%
\subsection{Decompositions, idempotents, and frames: $\mathcal{E}(\mathcal{X})$}
		\label{sec:frames}
	
We show how idempotents of $\Sym(b)$ parameterize $\perp$-decompositions of 
a Hermitian $k$-bilinear map $b:V\times V\to W$.  We start with the elementary
\begin{lemma}\label{lem:im-ker}
If $f\in\Sym(b)$ then $b(\im f,\ker f)=0$.
\end{lemma}
\begin{proof}
Let $u\in V$ and $v\in \ker f$.  Then $b(uf,v)=b(u,vf)=0$.
\end{proof}

By standard linear algebra, an idempotent $e$ in $\End V$ decomposes $V$ as  
$\im e \oplus \ker e$.  In light of \lemref{lem:im-ker}, if $e\in \Sym(b)$ 
then $b(\im e,\ker e)=0$, so we arrive at a $\perp$-decomposition $\{\ker f,\im f\}$. 

\begin{defn}\cite[pp.117-118]{Jacobson:Jordan}\label{def:idem-jord}
Let $J$ be a Jordan algebra.
\begin{enumerate}[(i)]
\item An \emph{idempotent} is an element $e$ in $J$ such that $e^2=e$.
It is \emph{proper} if it is neither $0$ nor $1$.
\item The Peirce-1-space of an idempotent $e$ is the subspace $JU_{e}$.
 The Peirce-0-space is $JU_{1-e}$.  These are Jordan algebras (in fact
 inner ideals) with identity $e$ and $1-e$,
 respectively (cf. \propref{prop:perp-idemp}).
\end{enumerate}
\end{defn}

\begin{prop}\label{prop:perp-idemp}
Let $e\in \End V$ with $e^2=e$, $E:=Ve$ and $F:=V(1-e)$.
\begin{enumerate}[(i)]
\item\label{id:1} $e\in \Sym(b)$ if, and only if, $b(E,F)=0$.
\item\label{id:2} If $e\in \Sym(b)$ then $\Sym(b)U_e$ is isomorphic as a Jordan
algebra to $\Sym(b_E)$ via the restriction of $f\in\Sym(b)U_e$ to $(fU_e)|_E:E\to E$.
\end{enumerate}
\end{prop}
\begin{proof}
$(\ref{id:1})$  \lemref{lem:im-ker} proves the forward direction. 
For the converse, as $b(E,F)=0$ it follows that $b(ue,v(1-e))=0=b(u(1-e),ve)$ 
for all $u,v\in V$.  Hence
\begin{multline*}
	b(ue,v)=b(ue,ve+v(1-e))=b(ue,ve)
		=b(ue+u(1-e),ve)=b(u,ve),
\end{multline*}
for all $u,v\in V$; thus, $e\in \Sym(b)$.

For (\ref{id:2}), note that $\Sym(b)U_e\subseteq e\Adj(b)e$ and so
$\Sym(b)U_e$ is faithfully represented in $\End E$ by restriction.
Furthermore, $b(uexe,v)=b(u,vexe)$ for all $u,v\in E$ and $x\in\Sym(b)$.
Thus the restriction of $\Sym(b)U_e$ is $\Sym(b_E)$.
\end{proof}

From \propref{prop:perp-idemp}.$(i)$ we see that $F=E^\perp$ (cf. \defref{def:perp}.$(ii)$).

\begin{defn}\cite[pp.117-118]{Jacobson:Jordan}
Let $J$ be a Jordan algebra.
\begin{enumerate}[(i)]
\item Two idempotents $e,f$ in $J$ are \emph{orthogonal} if $e\bullet f=fU_e=eU_f=0$ \cite[5.1]{Jacobson:Ark}.
\item An idempotent is \emph{primitive} if it is not the sum of two proper orthogonal idempotents.
\item A set of idempotents is \emph{supplementary} if the idempotents are pairwise 
orthogonal and sum to $1$.
\item A \emph{frame} $\mathcal{E}$ of $J$
is a set of primitive pairwise orthogonal idempotents which sum to $1$.
\end{enumerate}
\end{defn}

Idempotents in special Jordan algebras are idempotents in the associative algebra
as well.  If $e,f\in \Sym(b)$ then $e$ and $f$ are orthogonal idempotents
in $\Sym(b)$ if, and only if, they are orthogonal in $\Adj(b)$.  To see this, if 
$0=e\bullet f=\frac{1}{2}(ef+fe)$
and $efe=fU_e=0$ then $ef=ef+efe=e(ef+fe)=0$ and also $fe=0$.  If  
$ef=0=fe$ then $e\bullet f=\frac{1}{2}(ef+fe)=0$ (cf. \cite[p. 5.4]{Jacobson:Ark}).  
However, if $e$ is a primitive
idempotent in $\Sym(b)$ it need not follow that $e$ is primitive in $\Adj(b)$
since there may be orthogonal idempotents in $\Adj(b)$ which sum to $e$ but
do not lie in $\Sym(b)$.

The following definition is based on standard uses of idempotents in linear algebra.
\begin{defn}\label{def:perp-idemp}
Let $V$ be a vector space over $k$.
\begin{enumerate}[(i)]
\item Let $\mathcal{E}(\mathcal{Y})$ be the set of supplementary idempotents
parameterizing a given $\oplus$-decomposition $\mathcal{Y}$ of $V$.
\item Let $\mathcal{X}(\mathcal{F})$ be the $\oplus$-decomposition
arising from a set of supplementary idempotents $\mathcal{F}$ of $\End V$.
\end{enumerate}
\end{defn}

\begin{thm}\label{thm:perp-idemp}
Let $\mathcal{X}$ be a $\oplus$-decomposition of $V$ and let $\mathcal{E}
=\mathcal{E}(\mathcal{X})$.
\begin{enumerate}[(i)]
\item\label{fr:1} $\mathcal{E}(\mathcal{X})\subseteq \Sym(b)$ if, and only if,
$\mathcal{X}$ is a $\perp$-decomposition of $b$.
\item\label{fr:2} $\mathcal{X}$ is a fully refined $\perp$-decomposition 
if, and only if, $\mathcal{E}$ is a frame.
\item\label{fr:3} Let $\mathcal{X}$ be a $\perp$-decomposition.  
If $(\alpha,\hat{\alpha})\in \Isom^*(b)$, then 
$\mathcal{X}\alpha=\mathcal{X}(\mathcal{E}^{(\alpha,\hat{\alpha})})$
and $\mathcal{E}^{(\alpha,\hat{\alpha})}=\mathcal{E}(\mathcal{X}\alpha)$.
In particular, $\Isom^*(b)$ acts on the set of all frames of $\Sym(b)$.
\end{enumerate}
\end{thm}
\begin{proof}
Part (\ref{fr:1}) follows from \propref{prop:perp-idemp}.  Part (\ref{fr:2})
follows from observing that an idempotent $e\in\Sym(b)$ is primitive if,
and only if, $b_{Ve}$ is $\perp$-indecomposable.

For part (\ref{fr:3}), if $e\in\mathcal{E}$ and $x\in Ve\alpha$,
then $x(e^{(\alpha,\hat{\alpha})})=((x\alpha^{-1})e)\alpha=x\alpha^{-1}\alpha=x$. 
Therefore $V(e^{(\alpha,\hat{\alpha})})=Ve\alpha$.
\end{proof}

		%
		%
\subsection{Linking central decompositions, $\perp$-decompositions, 
		frames, and orthogonal bases: $\mathcal{H}_I$, $\mathcal{X}_I$, 
		$\mathcal{E}_I$, and $\mathcal{X}_{d(I)}$.}\label{sec:links}

We use the following notation repeatedly as a means to track the changes from 
$p$-groups, to bilinear maps, to $*$-algebras, to Hermitian forms, and then back.
As usual, we assume that $P$ has class $2$, exponent $p$, and $P'=Z(P)$.

Let $\mathcal{H}$ be a fully refined central decomposition of $P$, $\mathcal{X}$
a fully refined $\perp$-decomposition of $b:=\Bi(P)$, $\mathcal{E}$ a frame of
$J:=\Sym(b)$, $A :=\Adj(b)$, and $I\in \spec_0 A$ (that is, $I$ is a maximal 
$*$-ideal of $A$).  Define:
\begin{align}
	\mathcal{E}_I & = \{e\in\mathcal{E} : e\notin I\},\\
	\mathcal{X}_I & = \{ X\in\mathcal{X} : 
					e\in\mathcal{E}(\mathcal{X})_I, X=Ve\},\\
	\mathcal{H}_I & = \{ H\in\mathcal{H} : HP'/P'\in \Bi(\mathcal{H})_I \}.
\end{align}
Since $A/I\cong \Adj(d(I))$ for some non-degenerate Hermitian $C$-form 
$d:=d(I):U\times U\to C$, (\thmref{thm:*-structure}.$(iii)$)), it follows that
 $J/(I\intersect J)\cong \Sym(d)$.  Hence, $I\intersect J$ is a maximal ideal 
 of $J$ (\thmref{thm:herm-structure}.$(iv)$). Therefore, $\mathcal{E}_I$ 
 parameterizes a frame 
\begin{equation}
	\mathcal{E}_{J/(I\intersect J)}:=\{(I\intersect J)+e: e\in\mathcal{E}_I\}
\end{equation}
of $J/(I\intersect J)$.  Furthermore, this gives rise to a fully refined $\perp$-decomposition
\begin{equation}
	\mathcal{X}_{d(I)} := \{ Ue\tau : e\in\mathcal{E}_I\}
\end{equation}
of $d(I)$ where $\tau:A/I\to \Adj(d(I))$ is a $*$-isomorphism.  Certainly, 
$\mathcal{X}_{d(I)}$ depends on the choice of $\tau$ but we consider $\tau$
fixed.  This influences the definition of address in Section \ref{sec:address}.

\begin{prop}\label{prop:links}
Let $\mathcal{H}$ be a fully refined central decomposition of $P$, $\mathcal{X}:=\Bi(\mathcal{H})$, and $\mathcal{E}:=\mathcal{E}(\mathcal{X})$.  The sets
 $\mathcal{H}_I$, $\mathcal{X}_I$, $\mathcal{E}_I$, 
 $\mathcal{E}_{J/(I\intersect J)}$, and $\mathcal{X}_{d(I)}$ are in bijection.
\end{prop}
\begin{proof}
This follows from \thmref{thm:main-reduction}.$(ii)$, \thmref{thm:perp-idemp}.$(ii)$, \thmref{thm:herm-structure}.$(iii)$, and \propref{prop:orthogonal}.
\end{proof}

\begin{prop}\label{prop:partition}
For every fully refined central decomposition $\mathcal{H}$ of $P$ with $P'=Z(P)$, 
the set $\{\mathcal{H}_I : I\in\spec_0 \Adj(\Bi(P))\}$ partitions
$\mathcal{H}$.  Furthermore, $|\mathcal{H}_I |$ depends only on 
$P$ and $I\in\spec_0 \Adj(\Bi(P))$.
\end{prop}
\begin{proof}
By \propref{prop:links} we know $\mathcal{H}_I$ is in bijection with $\mathcal{E}_I$
for each maximal $*$-ideal of $\Adj(\Bi(P))$.  As $\mathcal{E}$ is partitioned by
$\mathcal{E}_I$, as $I$ ranges over the maximal $*$-ideals of $\Adj(\Bi(P))$, it follows
that $\{\mathcal{H}_I: I\in\spec_0\Adj(\Bi(P))\}$ partition $\mathcal{H}$.
\end{proof}

		%
		%
\subsection{All fully refined central decompositions have the same size}
		\label{sec:same-size}
		
We now prove the first part of \thmref{thm:central-count}.$(i)$ -- that fully refined central
decompositions of a $p$-group $P$ of exponent $p$ and class $2$ have the same size.

\begin{thm}\label{thm:central-III}
Let $P$ be a finite $p$-group of class 2 and exponent $p$ and $\mathcal{H}$
a fully refined central decomposition.  Let $Q:=\langle \mathcal{K}\rangle$, 
$\mathcal{K}:= \mathcal{H}-Z(\mathcal{H})$. Then $\mathcal{H}$ is partitioned into
\begin{equation}
	Z(\mathcal{H})\sqcup \{\mathcal{K}_I : I\in \spec_0  \Adj(\Bi(Q))\}.
\end{equation}
Furthermore, $|Z(\mathcal{H})|$ and $|\mathcal{K}|$ are uniquely determined by
$P$, and $|\mathcal{H}|$ is uniquely determined by $P$.
\end{thm}
\begin{proof}  
 By \lemref{lem:contract-1} we know $P=Q\oplus A$ with $A\leq Z(P)$ and
  $Q'=P'=Z(Q)$.  Furthermore, $|Z(\mathcal{H})|=\log_p |A|=\log_p [Z(P):P']$.  Therefore, \lemref{lem:contract-2} and \propref{prop:partition} complete the proof.
\end{proof}

		%
		%
\subsection{The five classical indecomposable families}

By \thmref{thm:perp-idemp}, a bilinear map $b$ has no proper $\perp$-decompositions 
if, and only if, $0$ and $1$ are the only idempotents of $\Sym(b)$.  But more 
can be said if $\Adj(b)$ is considered as well:
\begin{lemma}[Fitting's Lemma for bilinear maps]
If $b$ is $\perp$-indecomposable then, for every $x\in \Adj(b)$,
$T(x)=x+x^*$ is either invertible or nilpotent.  In particular,
every $x\in\Sym(b)$ is either invertible or nilpotent.
\end{lemma}
\begin{proof}
Set $y=x+x^*$ and note $y^r\in\Sym(b)$ for all $r\in\mathbb{N}$.  By 
Fitting's lemma there is some $r>0$ such that $V=\im y^r\oplus \ker y^r$.  
By \lemref{lem:im-ker}, $b(\im y^r,\ker y^r)=0$.  So we have a $\perp$-decomposition
of $b$.  Since $b$ is $\perp$-indecomposable, $y^r=0$ so that $y$ is nilpotent, 
or $\ker y^r=0$ and $\im y^r=V$ so that $y$ is invertible.
\end{proof}

\begin{thm}\cite[Theorem 2]{Osborn:jor-nil}\label{thm:local-jordan}
If $(A,*)$ is a $*$-algebra over a finite field of odd characteristic such that
$T(x)$ is either invertible or nilpotent for each $x\in A$, then $A/\rad A$ 
is an associative composition algebra.
\end{thm}

\begin{coro}\label{coro:indecomp}
For a $k$-bilinear map $b$ the following are equivalent:
\begin{enumerate}[(i)]
\item $b$ is $\perp$-indecomposable,
\item $\Sym(b)$ has only trivial idempotents,
\item $J/\rad J$ is isomorphic to a field extension of $k$. 
\item $A=\Adj(b)$ has $A/\rad A$ is isomorphic 
to an associative composition algebra.
\end{enumerate}
\end{coro}

\begin{thm}\label{thm:indecomps}
A $p$-group $P$ of class 2 and exponent $p$ is centrally indecomposable 
if, and only if, one of the following holds with 
$G:=C_{\Aut P}(Z(P))/ O_p (C_{\Aut P}(Z(P)))$:
\begin{description}
\item[Abelian] $|P|=p$,
\item[Orthogonal] $G\cong O(1,p^e)\cong \mathbb{Z}_2$ with $p\neq 3$, or
$p=3$ and \\
$C_{\Aut P\circ P}(P')/ O_p(C_{\Aut P\circ P}(P'))\cong \GO^\pm(2,3^e)$;
\item[Unitary] $G\cong U(1,p^e)\cong \mathbb{Z}_{p^e+1}$,
\item[Exchange] $|P|\neq p$ and 
$G\cong \GL(1,p^e)\cong \mathbb{Z}_{p^e-1}$, or
\item[Symplectic] $G\cong \Sp(2,p^e)\cong \SL(2,p^e)$;
\end{description}
for some $e>0$.  
\end{thm}
\begin{proof}
This follows from \corref{coro:indecomp}, \thmref{thm:isom-classical} and
\thmref{thm:main-reduction}.
\end{proof}
In Section \ref{sec:ex} we demonstrate that with the possible exception of the unitary type,  each of these types can occur.

%
%
\section{Isometry orbits of $\perp$-decompositions}\label{sec:main}

In this section we describe the orbits of $C_{\Aut P}(P')$ in its action on the set of fully refined central decompositions.  To do this, we define a computable $C_{\Aut P}(P')$-invariant for each fully refined central decomposition called its \emph{address}.  Then we
prove that any two fully refined central decompositions with the same address lie in the same orbit.

		%
		%
		\subsection{Addresses: $\mathcal{H}@$ and $\mathcal{X}@$}\label{sec:address}

\begin{defn}\label{def:address-simple}
Let $d:V\times V\to C$ be a non-degenerate Hermitian $C$-form.  
\begin{enumerate}[(i)]
\item Given a non-singular $x\in V$ (cf. \defref{def:non-singular}),
the \emph{address} of $X:=Cx$ is
\begin{equation*}
	X@:=d(x,x)N(C^\times),
\end{equation*}
as an element of $K^\times/N(C^\times)$.
\item $\mathcal{X}@ :=\{X@ : X\in\mathcal{X}\}$ (as a multiset indexed by
$\mathcal{X}$) for every fully refined $\perp$-decomposition $\mathcal{X}$ 
of $d$.
\end{enumerate}
\end{defn}

From \thmref{thm:Hurwitz} we know $N(C)=K$ if $C>K$ and therefore the addresses 
of non-singular points of a non-symmetric non-degenerate Hermitian $C$-form are all equal to $K^\times$.  However, for non-degenerate symmetric bilinear forms, the address is a coset of 
$(K^\times)^2$.
		
Let $d:V\times V\to K$ be a non-degenerate symmetric bilinear form.  
Fix $\omega\in K^\times-(K^\times)^2$.  Every address of a non-singular 
point of $V$ is either $[1]:=(K^\times)^2$ or $[\omega]:=\omega(K^\times)^2$.  If
$\mathcal{X}$ is an orthogonal basis of $d$, then for some $0\leq s\leq n$,
\begin{equation}\label{eq:address}
	\mathcal{X}@ =
		\{\overbrace{[1],\dots, [1]}^{n-s},\overbrace{[\omega],\dots,[\omega]}^s\},
		\qquad n=\dim V.
\end{equation}
We write $(n-s:s)$ for the address $\mathcal{X}@$.  

The discriminant of a Hermitian $C$-form $d$ is 
\begin{equation}
	\disc d=\prod_{X\in\mathcal{X}} X@
\end{equation}
as an element of $K^\times/N(C^\times)$ (cf. \eqref{def:disc}).  In particular,
if $d$ is symmetric then $\disc d=[\omega^s]$ (cf. \eqref{eq:address}).  
Otherwise we can regard the discriminant as trivial.

Let $P$ be a $p$-group $P$ of class $2$, exponent $p$, and $P'=Z(P)$.  Let $\mathcal{H}$ be a fully refined central decomposition of $P$, $\mathcal{X}:=\Bi(\mathcal{H})$, and $\mathcal{E}:=\mathcal{E}(\mathcal{X})$.  Using the notation of Section \ref{sec:links} and \propref{prop:links}, for each maximal $*$-ideal $I$ of $\Adj(\Bi(P))$, assign the address of $\mathcal{H}_I$, $\mathcal{X}_I$, $\mathcal{E}_I$, and $\mathcal{E}_{J/(I\intersect J)}$ as the address of $\mathcal{X}_{d(I)}$.  Finally, 
\begin{align}
	\mathcal{E}@ :=\{ (I,\mathcal{E}_I @) : I\in\spec_0 \Adj(\Bi(P))\},\\
	\mathcal{X}@ :=\{ (I,\mathcal{X}_I @) : I\in\spec_0 \Adj(\Bi(P))\},\\
	\mathcal{H}@ :=\{ (I,\mathcal{H}_I @) : I\in\spec_0 \Adj(\Bi(P))\}.
\end{align}
\begin{remark}
Recall that $\mathcal{X}_{d(I)}$ depends on the choice of non-degenerate
Hermitian $C$-form $d:=d(I):U\times U\to C$.  Any other choice is pseudo-isometric to $d$.  Suppose that $d':U'\times U'\to C$ is pseudo-isometric
to $d$ via $(\alpha,\beta)$.  Let $u\in U$ such that $d(u,u)\in K^\times$
(cf. \propref{prop:orthogonal}).  Then
\begin{equation}
	d(u,u)\beta=d(u\alpha,u\alpha)=\overline{d(u\alpha,u\alpha)}
		=\bar{\beta}d(u,u).
\end{equation}
Hence, $\beta=\bar{\beta}$; thus, $\beta\in K^\times$.  

The affect is that $\mathcal{X}_{d'} @ \beta=\mathcal{X}_d @$.  Therefore
the specific cosets in $K^\times/N(C^\times)$ are not significant, rather the
total number of members with the same address.  For finite fields the notation
$(n-s:s)$ is sufficient to record the address unambiguously.
\end{remark}

\begin{prop}\label{prop:address}
\begin{enumerate}[(i)]
\item If $\mathcal{X}$ is a fully refined $\perp$-decomposition of $b$ and 
$\varphi\in \Isom(b)$ then $X@=X\varphi @$ for all $X\in\mathcal{X}$.
\item If $\mathcal{H}$ is a fully refined central decomposition of $P$ and
$\varphi\in C_{\Aut P}(P')$ then $H@=H\varphi @$ for all $H\in\mathcal{H}$.
\end{enumerate}
\end{prop}
\begin{proof}
$(i)$.  Let $I\in\spec_0 \Adj(b)$ and $\Adj(b)/I\cong \Adj(d)$, 
$d:=d(I):U\times U\to C$.  By \propref{prop:groups}.$(ii)$, $\Isom(b)$ maps 
into $\Isom(d)$.  Let $X\in\mathcal{X}_I$ and $Cx$, $x\in U$, the corresponding 
member of $\mathcal{X}_{d(I)}$.  The address of $X$ is by definition the 
address of $Cx$.  As $d(x,x)=d(x\varphi,x\varphi)$ it follows that
$Cx\varphi @=Cx@$ and $X\varphi @=X@$.  $(ii)$.  This follows from $(i)$
and \thmref{thm:main-reduction}.
\end{proof}

		%
		%
		\subsection{Orbits of fully refined $\perp$-decompositions 
		of non-degenerate Hermitian $C$-forms}\label{sec:simples}

The theorems of this section are undoubtedly known, though with different terminology.

\begin{lemma}\label{lem:isom-inv}
Let $d:V\times V\to C$ be a non-degenerate Hermitian $C$-form and $\mathcal{X}$ a fully refined $\perp$-decompositions of $d$.  Then, for each $\varphi\in \Isom(d)$
there is a $\tau\in\Isom(d)$ which is a product of involutions and such that $X\varphi=X\tau$, for $X\in\mathcal{X}$.
\end{lemma}
\begin{proof} If the rank of $V$ is 1 then let $\tau=1$.  So assume the rank
is greater than $1$.  By 
\propref{prop:isom-type}, we have the four classical groups to consider.
The orthogonal groups are generated by reflections so take $\tau:=\varphi$.  In the exchange, unitary, and symplectic cases, the rank of $V$ excludes the case
$GF(q)^\times$, $\GU(1,q)$ and $\Sp(2,q)$.  Therefore the relevant symplectic groups are generated by their involutions and again $\tau:=\varphi$.  In the exchange and unitary cases the involutions generate a normal subgroup $N\geq \Isom(d)\intersect \SL(V)$.  Therefore $\varphi\equiv \mu \pmod{N}$ where $\mu$ is a diagonalizable.  Without loss of generality, 
$\mathcal{X}\mu=\mathcal{X}$, so take $\tau:=\mu^{-1}\varphi\in N$.
\end{proof}

\begin{thm}\label{thm:address}
Let $d:V\times V\to C$ be a non-degenerate Hermitian $C$-form and $\mathcal{X}$
and $\mathcal{Y}$ fully refined $\perp$-decompositions of $d$.  Then
there is an isometry $\varphi$ of $d$ such that $\mathcal{X}\varphi=\mathcal{Y}$
if, and only if, $\mathcal{X}@=\mathcal{Y}@$.  Indeed, if $\phi:\mathcal{X}\to \mathcal{Y}$ is a bijection where $X\phi@=X@$ for each $X\in\mathcal{X}$, then 
$\varphi$ can be taken as a product of involutions where $X\varphi=X\phi$,
for each $X\in\mathcal{X}$.
\end{thm}
\begin{proof} Suppose $\mathcal{X}\varphi=\mathcal{Y}$ for some $\varphi\in \Isom(d)$.
Given $X\in \mathcal{X}$, $d(x\varphi,x\varphi)=d(x,x)$ for each $x\in X$; hence,
$X@$ equals $X\varphi@$.  Thus, the addresses of $\mathcal{X}$ and $\mathcal{Y}$ agree.

For the converse, suppose we have a bijection $\phi$ as described above.  Fix 
generators $x$ and $y_x$ for $X=Cx\in\mathcal{X}$ and $X\phi=Cy_x\in\mathcal{Y}$, respectively.  By assumption, there is an $s_x\in C^\times$ such that 
$d(x,x)=N(s_x) d(y_x,y_x)$.

Define $\varphi:V\to V$ by $x\varphi=s_x y_x$ for each 
$X=Cx \in \mathcal{X}$.  It follows that 
$d(x\varphi,x\varphi)=N(s_x)d(y_x,y_x)=d(x,x)$ for all 
$X=Cx\in \mathcal{X}$; thus, $\varphi\in \Isom(d)$.
Furthermore, $\mathcal{X}\varphi=\mathcal{Y}$ and $X\varphi=X\phi$.
To convert $\varphi$ into a product of involutions, invoke \lemref{lem:isom-inv}.
\end{proof}

We also require the following version of transitivity as well.
\begin{thm}\label{thm:singles}
Let $d:V\times V\to C$ be a non-degenerate Hermitian $C$-form.  If $X,Y\in V$ 
are non-singular points (\defref{def:non-singular}), then $X\varphi=Y$ 
for some $\varphi\in\Isom(d)$ if, and only if, $X@=Y@$.
\end{thm}
\begin{proof}
If $X\varphi = Y$ then $X@=Y@$.  

For the reverse direction suppose that $X@=Y@$.
Since $X@\disc d_{X^\perp}=\disc d=Y@\disc d_{Y^\perp}$, it
follows that $\disc d_{X^\perp}=\disc d_{Y^\perp}$.  By 
\eqref{eq:dot-prod} for the symmetric case and \propref{prop:orthogonal} for all other cases,
there are orthogonal bases $\mathcal{X}'$ of $d_{X^\perp}$ and $\mathcal{Y}'$ of 
$d_{Y^\perp}$ such that $\mathcal{X}'@=\{[1],\dots,[1],[\disc d_{X^\perp}]\}$ and 
$\mathcal{X}'@=\{[1],\dots,[1],[\disc d_{Y^\perp}]\}$.  Set $\mathcal{X}=\{X\}\sqcup\mathcal{X}'$ and $\mathcal{Y}:=\{Y\}\sqcup\mathcal{Y}'$.  Then $\mathcal{X}$ and
$\mathcal{Y}$ are fully refined $\perp$-decompositions of $d$.  Furthermore,
\begin{equation*}
\mathcal{X}@ = \{X@,[1],\dots,[1],[\disc d_{X^\perp}]\}
	= \{Y@,[1],\dots,[1],[\disc d_{Y^\perp}]\}
	=\mathcal{Y} @.
\end{equation*}
Therefore, by \thmref{thm:address}, there is a
$\varphi\in \Isom(d)$ such that $\mathcal{X}\tau=\mathcal{Y}$
and $X\varphi=Y$.
\end{proof}

		%
		%
\subsection{Orbits of frames in Jordan algebras}\label{sec:radicals}

In this section we determine the orbits of $\Isom(b)$ acting on fully refined $\perp$-decompositions of $b$, for an arbitrary Hermitian bilinear map $b:V\times V\to W$.  To do this we use frames, radicals, and  the semi-simple structure of the Jordan algebra $\Sym(b)$.
We caution that we make frequent use of Sections \ref{sec:Jordan} and \ref{sec:frames}, at times without specific reference.

Suppose $\mathcal{X}$ is a fully refined $\perp$-decomposition of $b$.  By \thmref{thm:perp-idemp}, $\mathcal{E}:=\mathcal{E}(\mathcal{X})$ is a frame of $\Sym(b)$.
We also know that $\Isom(b)$ acts on $\Sym(b)$ by conjugation (\thmref{thm:jordan})
and that $\mathcal{E}^\varphi = \mathcal{E}(\mathcal{X}\varphi)$ for each
$\varphi\in \Isom(b)$ (\thmref{thm:perp-idemp}).  Therefore, it suffices to work with the orbits of frames of $\Sym(b)$ under the action of $\Isom(b)$.  To make use of the Jordan algebra we also translate the action of $\Isom(b)$ into Jordan automorphisms of $\Sym(b)$ in the following way.

By \propref{prop:groups}.$(ii)$, every isometry $\varphi$ has the defining property $\varphi \varphi^*=1$. Hence, $\varphi\in \Sym(b)\intersect \Isom(b)$ if, and only if, $\varphi^2=1$.

\begin{defn}\label{def:inv}
Define $\Inv(J)=\langle U_x:x\in J, x^2=1\rangle\leq \GL(J)$ for a special Jordan algebra $J$.
\end{defn}
We consider only those Jordan algebras $J$ which are subalgebras or quotient algebras of a special Hermitian Jordan algebra such as $\Sym(b)$.  Note that if $x\in J$ with $x^2=1$ then $yU_x=x^{-1}yx=y^x$ for all $y\in J$.  Therefore each element of $\Inv(J)$ acts both as a product of $U$-operators and as conjugation.  So $\Inv(J)$ is a group of automorphisms of $J$ built from elements of $J$.

\begin{remark}\label{rem:inv-isom}
The group $\Inv(\Sym(b))$ is not contained in $\Isom(b)$ and we are careful to distinguish the action on $J:=\Sym(b)$ by the two groups as follows:  if $\varphi\in \Isom(b)$ then write $y^\varphi$ (cf. \propref{prop:groups}.$(i)$), and if $\varphi\in \Inv(J)$ then use the usual function notation $y\varphi$, for $y\in J$.  However, $\Inv(\Sym(b))$ embeds in $\Isom(b)$ by extending $U_x\mapsto x$, $x\in \Sym(b)$, $x^2=1$.
\end{remark}

By \defref{def:idem-jord}, if $e\in J$ is an idempotent then $JU_e=eJe$ is a
subalgebra with identity $e$.
\begin{prop}\label{prop:extend-inv}
Let $e$ be an idempotent in $J$.  Then $\Inv(JU_e)$ embeds in $\Inv(J)$
acting as the identity on $JU_{1-e}$.
\end{prop}
\begin{proof}
It suffices to extend the generators of $\Inv(JU_e)$ to $J$.  Let $v\in JU_e$
with $v^2=e$.  Set $u:=(1-e)+v\in J$.  As $v=vU_e=eve$ it follows that 
$u^2=(1-e)^2+(1-e)eve+eve(1-e)+v^2=1$, so $U_u\in\Inv(J)$.  Furthermore,
if $x\in JU_e$, then $xU_u=xU_eU_u=((1-e)+v)exe((1-e)+v)=xU_v$.  Finally, if 
$x\in JU_{1-e}$, then $xU_u=xU_{1-e}U_u=((1-e)+v)(1-e)x(1-e)((1-e)+v)=x$.
\end{proof}

\begin{lemma}\cite[III.7,Lemma 4]{Jacobson:Jordan}\label{lem:inv-lift}
Let $N$ be a nil ideal in $J$.  If $N+u\in J/N$ with $u^2-1\in N$, 
then there is a $v\in J$ such that $N+u=N+v$ and $v^2=1$.
\end{lemma}

\begin{prop}\label{prop:lift-inv}
\begin{enumerate}[(i)]
\item If $\varphi\in \Inv(J)$ then $(\rad J)\varphi = \rad J$ and 
$\varphi|_{J/\rad J}\in \Inv(J/\rad J)$.
\item Suppose $N\normaleq J$ and $N$ is 
nil (in particular for $N\subseteq \rad J$).  Then for each
$\hat{\varphi}\in\Inv(J/N)$ there is a $\varphi\in\Inv(J)$ such 
that $\varphi|_{J/N}=\hat{\varphi}$.  
\end{enumerate}
\end{prop}
\begin{proof}
(i) $\Inv(J)$ is a subgroup of the automorphism group of $J$ and so maximal 
inner ideals are mapped to 
maximal inner ideals and the radical is preserved.  Since involutions of $J$ 
are sent to involutions of $J/\rad J$, it follows that 
$\Inv(J)|_{J/\rad J}\leq \Inv(J/\rad J)$.

(ii) By definition $\Inv(J/N)$ is generated by the $U_{\hat{v}}$ for which $\hat{v}$
is an involution of $J/N$.  For each $\hat{v}$, by \lemref{lem:inv-lift} there 
is an involution $v\in J$ such that $\hat{v}=v+N$.  Thus 
$U_{\hat{v}}=U_{v+N}=(U_v)|_{J/N}$, $U_v\in \Inv(J)$.
\end{proof}

\begin{lemma}\label{lem:nil-idemp}
Let $e,e'\in J$ be orthogonal idempotents.  If $z\in J$ such that
$z^2=0$ and $e+z$ is an idempotent, then there is a $v\in J$ such that
(i) $v^2=1$, (ii) $eU_v=e+z$ and (iii) $e'U_v=e'-2e'\bullet z+e'U_z$.
\end{lemma}
\begin{proof}
Let $v=1-2e-z$.

$(i)$.  Since $e+z=(e+z)^2=e+ez+ze$ it follows that $z=ez+ze$.  
Hence,  $v^2=1-4e+4e^2-2z + 2ez+2ze+z^2=1$. 
For $(ii)$ note that $0=z^2=ez^2+zez$ so that $zez=0$.  Thus,
\begin{align*}
(1-2e-z)e(1-2e-z)
	& = ((1-2e-z)e)(e(1-2e-z))\\
	& = (e+ze)(e+ez)
	= e+ez+ze
	=e +z.
\end{align*}
So $eU_v=e+z$.  Finally for (iii):
\begin{equation*}
e'U_v 
	= (1-2e-z)e'(1-2e-z)
	=(e'-ze')(e'-e'z)
	=e'-2e'\bullet z+e'U_z.
\end{equation*}
\end{proof}

\begin{lemma}\label{lem:nil-fix}
Let $N$ be an ideal in $J$ such that $N^2=0$.
If $\mathcal{E}$ and $\mathcal{F}$ are both sets of supplementary idempotents of $J$ such that $\mathcal{E}\equiv \mathcal{F}\pmod{N}$, then there is  $\varphi\in \Inv(J)$
such that $\mathcal{E}\varphi=\mathcal{F}$.
\end{lemma}
\begin{proof}
Take $e\in\mathcal{E}-\mathcal{F}$ and $f=e+z\in\mathcal{F}$, $z\in N$ so that $z^2=0$.
By \lemref{lem:nil-idemp}.(i,ii), there is an involution $v\in J$ such that $eU_v=e+z=f$.
Hence, $\mathcal{E}':=\mathcal{E}U_v$ is a supplementary set of idempotents of $J$.  By 
\lemref{lem:nil-idemp}(iii), $\mathcal{E}'\equiv \mathcal{E} \pmod{N}$ so that 
$\mathcal{E}'\equiv \mathcal{F}\pmod{N}$.  Also, $f\in\mathcal{E}'\intersect \mathcal{F}$.

We now induct on the size of $\mathcal{E}$.  In the base case $\mathcal{E}=\{e\}$ 
and $\mathcal{F}=\{f\}$, so $\mathcal{E}U_v=\mathcal{E}'=\mathcal{F}$.  Otherwise,
as $\mathcal{E}'$ is a set of supplementary idempotents, for all $e'\in\mathcal{E}'-\{f\}$, $e'U_{1-f}=e'$ so $\mathcal{E}'-\{f\}=\mathcal{E}'U_{1-f}-\{0\}$ and similarly
$\mathcal{F}-\{f\}=\mathcal{F}U_{1-f}-\{0\}$.  So $\mathcal{E}'-\{f\}$ and 
$\mathcal{F}-\{f\}$ are both sets of supplementary idempotents in $JU_{1-f}$, where 
$\mathcal{E}'-\{f\}\equiv\mathcal{F}-\{f\} \pmod{NU_{1-f}}$. 
By induction there is a  $\tau'\in \Inv(JU_{1-f})$ such that 
$(\mathcal{E}'-\{f\})\tau'=\mathcal{F}-\{f\}$.  By \propref{prop:extend-inv} 
there is a $\tau\in\Inv(J)$ extending $\tau'$ to $J$ so that $\tau$ is 
the identity on $JU_f$.  So $\mathcal{E}'\tau=\mathcal{F}$.  Thus 
$U_v \tau\in \Inv(J)$ with $\mathcal{E}U_v\tau=\mathcal{F}$.
\end{proof}

\begin{prop}\label{prop:reduce-semi-simple}
Two sets of supplementary idempotents of $J$ are equivalent under the action of $\Inv(J)$ if, and only if, their images in $J/\rad J$ are equivalent under the action of $\Inv(J/\rad J)$.
\end{prop}
\begin{proof}
The forward direction follows from \propref{prop:lift-inv}.$(i)$.  For the converse, let 
$\mathcal{E}$ and $\mathcal{F}$ be sets of supplementary idempotents of $J$ such that 
$\mathcal{E}\tilde{\varphi}\equiv \mathcal{F} \pmod{\rad J}$ for some
$\tilde{\varphi}\in \Inv(J/\rad J)$.  By \propref{prop:lift-inv}.$(ii)$ we can replace $\tilde{\varphi}$ with some $\varphi\in \Inv(J)$.  

We will induct on the dimension of $\rad J$. In the base case $\rad J=0$ and the
result is clear.  Now suppose $N:=\rad J>0$.  By \cite[Lemma V.2.2]{Jacobson:Jordan} 
there is an ideal $M$ of $J$ such that $N^2\subseteq M\subset N$.  Then $\mathcal{E}\varphi \equiv \mathcal{F} \pmod{N/M}$ in $J/M$ and $(N/M)^2=0$, so by \lemref{lem:nil-fix} there is a $\tilde{\mu}\in \Inv(J/M)$ such that $\mathcal{E}\varphi \tilde{\mu}\equiv\mathcal{F} \pmod{M}$.  By \propref{prop:lift-inv}.$(ii)$, $\hat{\mu}$ lifts to some $\mu\in \Inv(J)$ such that $\mathcal{E}\varphi \mu\equiv \mathcal{F}\pmod{M}$.  As $M$ is a nil ideal properly contained in $N$, using $M$ in the r\^{o}le of $N$ and inducting we find a $\tau\in\Inv(J)$ such that $\mathcal{E}\varphi\mu\tau=\mathcal{F}$.
\end{proof}

\begin{thm}\label{thm:inv-trans}
$\Inv(J)$ is transitive on the set of frames of $\Sym(b)$ which have any given address.
\end{thm}
\begin{proof}
By \propref{prop:reduce-semi-simple} we may assume $\rad J=0$.
By \thmref{thm:herm-structure}.$(ii,iii)$, $J$ is the direct product
of a uniquely determined set $\mathcal{M}$ of  
simple Jordan matrix algebras.
If $e$ is a primitive idempotent of $J$ then $eJe$ is a minimal inner ideal
of $J$ (cf. \cite[Theorem 1.III]{Jacobson:Jordan}), and so $e$ lies in a minimal 
ideal of $J$, thus in a unique simple direct factor of $J$.  Hence, if $\mathcal{E}$ 
is a frame of $J$ then $M\intersect \mathcal{E}$ is a frame of $M$, for each 
$M\in\mathcal{M}$.  Furthermore, $\Inv(J)$ restricts to 
$\Inv(M)$ for each $M\in\mathcal{M}$.  
Thus \corref{thm:address} and \remref{rem:inv-isom} show that $\Inv(J)$ is 
transitive on frames with the same address.
\end{proof}

\begin{coro}\label{coro:main}
\begin{enumerate}[(i)]
\item $\Isom(b)$ acts transitively on the set of fully refined $\perp$-de\-comp\-o\-si\-tions with a given address.
\item If $P$ is a $p$-group of class $2$, exponent $p$, and $P'=Z(P)$, then $C_{\Aut P}(P')$ acts transitively on the set of fully refined central decompositions with a given address.
\end{enumerate}
\end{coro}
\begin{proof}
$(i)$.  This follows from \thmref{thm:inv-trans} and \remref{rem:inv-isom}.
$(ii)$.  This follows form part $(i)$ and \thmref{thm:main-reduction}.
\end{proof}

\begin{coro}\label{coro:isom-address}
Let $b:V\times V\to W$ be a non-degenerate Hermitian bilinear map.  Suppose that
$X$ and $Y$ are two $\perp$-factors of $b$.
\begin{enumerate}[(i)]
\item   Then there is a $\varphi\in \Isom(b)$
such that $X\varphi=Y$ if, and only if, $X@=Y@$ (which includes $X\in\mathcal{X}_I$,
$Y\in\mathcal{Y}_I$ for the same maximal $*$-ideal $I$ of $\Adj(b)$).
\item $b_X$ is isometric to $b_Y$ if, and only if, $X@=Y@$. 
\item Let $P$ be a $p$-group of class $2$, exponent $p$, and $P'=Z(P)$ with 
centrally indecomposable subgroups $H$ and $K$.  Then there is a 
$\varphi\in C_{\Aut P}(P')$ such that $H\varphi=K$ if, and only if,
$H@=K@$.
\end{enumerate}
\end{coro}
\begin{proof}
The forward direction of $(i)$ and $(ii)$ are clear.  For the reverse, use \thmref{thm:singles}, \lemref{lem:isom-inv}, \remref{rem:inv-isom}, and 
\propref{prop:lift-inv}.$(ii)$ to arrange for 
$\mathcal{E}(\{X,X^\perp\})\equiv\mathcal{E}(\{Y,Y^\perp\})$.  Then
\propref{prop:reduce-semi-simple} completes the proof.  $(iii)$.  This follows from
$(ii)$ and \thmref{thm:main-reduction}.
\end{proof}

%
%
\section{Semi-refinements and proof of \thmref{thm:central-count}.$(i)$}\label{sec:orbits}

By \thmref{coro:main}.$(i)$, any two fully refined $\perp$-decompositions with the
same address have the same multiset of isometry types.  This section is concerned with strengthening this result by involving pseudo-isometries in order to prove \thmref{thm:central-count}.$(i)$.

		%
		%
\subsection{The orthogonal bases of symmetric bilinear forms}
		
Let $d:V\times V\to K$ be a non-degenerate symmetric bilinear form and
recall the notation $(n-s:s)$ for addresses, given in Section \ref{sec:address}.  

\begin{lemma}\label{lem:pair}
If $\mathcal{X}$ and $\mathcal{Y}$ are fully refined $\perp$-decompositions of $d$ with 
$\mathcal{X}@=(n-s:s)$ and $\mathcal{Y}@=(n-r:r)$, then $2|s-r$.
\end{lemma}
\begin{proof}
Recall that the discriminant is independent of the basis of $V$.  Hence,
we have $[\omega^s]=\disc d=[\omega^r]$ so that
$\omega^{s-r}\equiv 1 \pmod{(K^\times)^2}$ and $2|s-r$.
\end{proof}

\begin{thm}\label{thm:coupling}
Let $\mathcal{X}$ be a fully refined $\perp$-decomposition with address $(n-r:r)$.
There is an involution $\rho\in \Isom(d)$ where $\mathcal{X}\rho=\mathcal{X}$
and such that, if $S:=\{X\in\mathcal{X} : X\rho=X\}$ then
\begin{enumerate}[(i)]
\item if $|\mathcal{X}|=2m+1$ then $S=\{X\}$ with $X@ =\disc d$,
\item if $|\mathcal{X}|=2m$ and $\disc d=[\omega]$ then $S=\{X,X'\}$
with $X@=[1]$, $X'@=[\omega]$,
\item if $|\mathcal{X}|=2m$ and $\disc d=[1]$ then $S=\emptyset$,
\item and for each $0\leq s\leq n$, where $2|r-s$, there is a fully refined 
$\perp$-decomposition $\mathcal{Y}$ where 
\begin{enumerate}[(a)]
\item $\mathcal{Y}@=(n-s:s)$,
\item $\langle X,X\rho\rangle 
= \langle \mathcal{Y}\intersect \langle X,X\rho\rangle\rangle$
for each $X\in\mathcal{X}$.
\end{enumerate}
\end{enumerate}
\end{thm}
\begin{proof}
We proceed by induction on the size of $\mathcal{X}$.

If $\mathcal{X}=\{X\}$ then let $\rho=1$ and $\mathcal{Y}=\mathcal{X}$.  
Hence $S=\mathcal{X}$ and $\disc d=X@$, as required by $(i)$. Also $(iv)$ is
satisfied trivially.

If $\mathcal{X}=\{X,X'\}$, $X\neq X'$  then $\disc d=X @ X'@$.

If $X@\neq X'@$ then take $\rho=1$ and $\mathcal{Y}=S=\mathcal{X}$ 
and up to relabeling, $(ii)$ is satisfied.  Once again, $(iv)$ is satisfied trivially
as $s=r=1$.

Suppose that $X@=X'@$.  By \thmref{thm:address} there is a $\rho\in \Isom(d)$
where $X\rho=X'$ and $X'\rho =X$, and indeed we may take $\rho^2=1$.
Notice $S=\emptyset$ and $\disc d=[1]$, as required by $(iii)$.  For $(iv)$, either
$s=r$ and we let $\mathcal{Y}=\mathcal{X}$ or $s=2-r$.  By \lemref{lem:line} there
is $(\varphi,\omega)\in \Isom^*(d)$; hence, $\mathcal{Y}:=\mathcal{X}\varphi$
satisfies $(iv)$.

If $n=|\mathcal{X}|> 2$ then there are distinct $X,X'\in \mathcal{X}$ with 
$X@=X'@$.  By induction on $\mathcal{Z}:=\mathcal{X}-\{X,X'\}$ we have an isometry
$\tau$ of $d_{\langle \mathcal{Z}\rangle}$ which permutes $\mathcal{Z}$.  We also 
induct on $\{X,X'\}$ to locate an involution $\mu\in \Isom(d_{\langle X,X'\rangle})$ 
such that $X\mu=X'$.  Set $\rho=\tau\oplus \mu\in\Isom(d)$.  Hence, $\rho^2=1$ 
and permutes $\mathcal{X}$.  Moreover, 
$\{X\in\mathcal{X} : X\rho = X\}=S=\{Z\in\mathcal{Z} : Z\tau=Z\}$ and 
$\disc d=X@ X'@\disc d_{\langle\mathcal{Z}\rangle}
=\disc d_{\langle \mathcal{Z}\rangle}$.  Therefore, each case of $S$ 
is satisfied for $\mathcal{X}$ with $\rho$ as it is satisfied for $\mathcal{Z}$ with $\tau$.
Therefore $\rho$ satisfies $(i)$, $(ii)$, and $(iii)$.

For $(iv)$, let $2|r-s$.   First assume $s\geq 2$.  From the induction on $\mathcal{Z}$ 
there is a fully refined $\perp$-decomposition $\mathcal{W}$ of $\langle \mathcal{Z}\rangle$ 
of address $(n-2:s-2)$ such that $\langle Z,Z\tau\rangle =\langle \mathcal{Y}\intersect
\langle Z,Z\rho\rangle\rangle$ for each $Z\in\mathcal{Z}$.  If $X@=[\omega]$ then 
set $\mathcal{Y}=\mathcal{W}\sqcup\{X,X'\}$ to complete $(iv)$.  If $X@=[1]$ then
use $(\varphi,\omega)\in \Isom^*(d_{\langle X,X'\rangle})$ from \lemref{lem:line}
and set $\mathcal{Y}:=\mathcal{W}\sqcup \{X\varphi,X'\varphi\}$.  Finally, if $s<2$
then take $\mathcal{W}$ to have address $(n-2:s)$ and define
$\mathcal{Y}:=\mathcal{W}\sqcup\{X\varphi,X'\varphi\}$ if $X@=[\omega]$, and
$\mathcal{Y}:=\mathcal{W}\sqcup\{X,X'\}$ otherwise.
\end{proof}

\begin{coro}\label{coro:count}
The set of addresses of orthogonal bases of $d$ is 
	\[\left\{(n-(c+2k):c+2k): 0\leq k\leq \frac{n-c}{2}\right\}\]
where $\disc d=[\omega^c]$, $c=0,1$.  In particular, there are 
$1+\lfloor\frac{n-c}{2}\rfloor$ addresses.
\end{coro}
\begin{proof}
From \thmref{thm:coupling}.$(iv)$, there is a fully refined $\perp$-decomposition
of $d$ for each address in the set.  By \lemref{lem:pair}, these are the possible 
addresses of $d$.
\end{proof}

\begin{coro}\label{coro:pseudo-orbits}
Let $d:V\times V\to K$ be a non-degenerate symmetric bilinear form with $n=\dim V$
and let $\mathcal{X}$ and $\mathcal{Y}$ be orthogonal bases with addresses $(n-s:s)$ and
$(n-r:r)$, respectively.
\begin{enumerate}[(i)]
\item If $n$ is odd then $\mathcal{X}\varphi=\mathcal{Y}$ for some 
$(\varphi,\hat{\varphi})\in \Isom^*(d)$ if, and only if, $s=r$.
\item If $n$ is even then $\mathcal{X}\varphi=\mathcal{Y}$ for some 
$(\varphi,\hat{\varphi})\in \Isom^*(d)$ if, and only if, $s=r$ or $s=n-r$.
\end{enumerate}
\end{coro}
\begin{proof}
Let $\mathcal{X}\varphi=\mathcal{Y}$.  Then as $\hat{\varphi}\in K^\times$,
$\hat{\varphi}\equiv 1$ or $\omega$ $\pmod{(K^\times)^2}$.  If $x\in\mathcal{X}$,
then
	\[X\varphi @\equiv d(x\varphi,x\varphi)\equiv d(x,x)\hat{\varphi}
		\equiv X@\hat{\varphi} \pmod{(K^\times)^2},\qquad X=\langle x\rangle.\]
Thus $\mathcal{Y}@=\mathcal{X}@\hat{\varphi}$.  If $\hat{\varphi} \equiv 1 
\pmod{(K^\times)^2}$ then $s=r$.  If $\hat{\varphi}\equiv \omega$ then
$s=n-r$, and
\[(\disc d)[\omega^n]=\prod_{X \in\mathcal{X}} X@\hat{\varphi}
	=\prod_{Y\in\mathcal{Y}} Y@=\disc d.\]
So, $2|n$.  This completes $(i)$.

For the converse, by \thmref{thm:address} it remains only to consider $s=n-r$,
which means $\mathcal{X}@=\mathcal{Y}@[\omega]$, and from above also $n=2m$.
By \propref{prop:dot-isom-*}.$(ii)$ there is a $(\varphi,\omega)\in\Isom^*(d)$.
Therefore $\mathcal{X}\varphi @=\mathcal{Y}@$.  By \thmref{thm:address}
there is a $\tau\in \Isom(d)$ such that $\mathcal{X}\varphi\tau=\mathcal{Y}$.
This completes $(ii)$.
\end{proof}

		%
		%
\subsection{Semi-refinements}\label{sec:final-proof}

\begin{defn}\label{def:semi-refined-bi}
A bilinear map $b:V\times V\to W$ is $\perp$-\emph{semi-indecomposable} if it is either $\perp$-indecomposable or $b$ has orthogonal type with a fully refined $\perp$-decomposition 
$\{X,Y\}$ such that $X@=Y@$.

A $\perp$-decomposition is \emph{semi-refined} if it consists of $\perp$-semi-indecomposables and it has no coarser $\perp$-decomposition consisting of $\perp$-semi-indecomposables.
\end{defn}

\begin{remark}\label{rem:c-perp-c}
Suppose that $b$ is a $\perp$-semi-indecomposable bilinear map which is not 
$\perp$-in\-de\-com\-pos\-a\-ble.  Hence, we have a fully refined $\perp$-decomposition 
$\{X,Y\}$ of $b$ with $X@=Y@$.  By \corref{coro:isom-address}.$(ii)$, this is equivalent
to having an isometry $\varphi\in \Isom(b)$ in which $X\varphi=Y$.  Thus
$b_{X}$ is isometric to $b_{Y}$.  Hence, if $c:=b_{X}$ then $b$ is isometric
to $c\perp c$.
\end{remark}

\begin{thm}\label{thm:semi-refined}
Let $b$ be a non-degenerate Hermitian bilinear map.
\begin{enumerate}[(i)]
\item Given a semi-refined $\perp$-decomposition $\mathcal{Z}$ and any fully refined
$\perp$-decomposition $\mathcal{X}$, there is a fully refined $\perp$-decomposition
$\mathcal{Y}$ with $\mathcal{X}@=\mathcal{Y}@$ and 
\begin{equation*}
	\mathcal{Z}=\mathcal{Y}^{[\rho]} :=\{\langle Y,Y\rho\rangle : Y\in\mathcal{Y}\},
\end{equation*}
where $\rho\in \Isom(b)$ is an involution. In particular,
 $|\mathcal{Z}|\geq |\mathcal{X}|/2$.
\item  $\Isom(b)$ acts transitively on the set of semi-refined $\perp$-decompositions.
\item  Every fully refined $\perp$-decomposition of a bilinear map $b$ determines
a semi-refined $\perp$-decomposition (as in $(i)$).  In particular, semi-refined 
$\perp$-de\-com\-pos\-i\-tions exist.
\end{enumerate}
\end{thm}
\begin{proof}
$(i)$.  The idempotents associated to a semi-indecomposable $b_Z$, $Z\in\mathcal{Z}$,
project to the same simple factor of $\Adj(b)$.  By \propref{prop:partition}, 
$\{\mathcal{Z}_I :I\normal\Adj(b)$ a maximal $*$-ideal $\}$ partitions 
$\mathcal{Z}$.  Hence, it suffices to consider $\mathcal{Z}_I$ for a fixed maximal $*$-ideal $I$ of $\Adj(b)$.

For each $Z\in \mathcal{Z}_I$, either $b_Z$ is $\perp$-in\-de\-com\-pos\-a\-ble or
it has a $\perp$-decomposition of size 2 whose members have equal addresses.  As 
$\mathcal{Z}_I$ is semi-refined, the set $S=\{Z\in\mathcal{Z}:b_Z$ is
$\perp$-indecomposable$\}$ has size $1$ if $|\mathcal{Z}_I|$ is odd, or
size 2 with $S=\{Y,Y'\}$ and $Y@\neq Y'@$, or $S=\emptyset$.  
It follows that $\mathcal{Z}_I$ is determines a fully refined $\perp$-decomposition
\begin{equation*}
	\mathcal{Y}_I:=\left(\sqcup_{Z\in\mathcal{Z}_I}\{Y_Z,Y'_Z\}\right)
		\sqcup S
\end{equation*}
in which $Y_Z@=Y'_Z@$ and $Z=\langle Y_Z,Y'_Z\rangle$, for each 
$Z\in\mathcal{Z}-S$.  By \thmref{thm:coupling} and \lemref{lem:inv-lift}, 
there is an involution $\rho \in \Isom(b)$ for which 
$\mathcal{Y}^{[\rho]}=\mathcal{Z}$ and furthermore, such that 
$\mathcal{X}_I@=\mathcal{Y}_I@$.

$(ii)$.  Let $\mathcal{W}$ be another semi-refined $\perp$-decomposition of $b$.
As in $(i)$ we know $\mathcal{W}=\mathcal{U}^{[\tau]}$ where $\mathcal{U}$
is fully refined and has address equal to that of $\mathcal{Y}$.  So we may define
a bijection $\phi:\mathcal{Y}\to\mathcal{U}$ such that $Y@=Y\phi @$ and
$Y\rho\phi=Y\phi\tau$, for each $Y\in\mathcal{Y}$.  By \corref{coro:main},
$\phi$ induces a $\varphi\in \Isom(b)$ such that 
$Y\varphi=Y\phi$ so that $\mathcal{Y}\varphi=\mathcal{U}$ and $\mathcal{Y}^{[\rho]}\varphi=\mathcal{U}^{[\tau]}$.

$(iii)$.  Let $\mathcal{X}$ be a fully refined $\perp$-decompositions.  
From $(i)$, any semi-refined $\perp$-decomposition can be fully
refined to have the same address of $\mathcal{X}$.  By $(ii)$ is this unique
up to an isometry.  Therefore it remains only to prove that there is a 
semi-refined $\perp$-decomposition.  This follows from \thmref{thm:coupling}.
\end{proof}

\begin{defn}\label{def:semi-refined}
A $p$-group $P$ of class $2$ and exponent $p$ is centrally 
\emph{semi-in\-de\-com\-pos\-a\-ble} if it is either centrally 
indecomposable or $P=H\circ H$ where $H$ is centrally indecomposable
of orthogonal type (cf. \eqref{eq:canonical-cent}).

A central decomposition is \emph{semi-refined} if it consists of centrally 
semi-in\-de\-com\-pos\-a\-ble subgroups and it has no coarser central decomposition consisting of centrally semi-indecomposable subgroups.
\end{defn}

\begin{remark}\label{rem:H-circ-H}
If $P$ is centrally semi-indecomposable and not centrally decomposable then 
$P=H\circ H$ where $H$ is centrally indecomposable.  Thus $\Bi(P)=\Bi(H)\perp \Bi(H)$.
As in \remref{rem:c-perp-c}, this is equivalent to having a fully refined
central decomposition $\{H,K\}$ of $P$ where $H@=K@$.
\end{remark}

\begin{coro}\label{coro:central-II}
Every fully refined central decomposition $\mathcal{H}$ of a $p$-group $P$ of class 
$2$, exponent $p$, and $P'=Z(P)$, generates a semi-refined central decomposition
\begin{equation*}
	\mathcal{H}^{[\rho]} := \{ \langle H,H\rho\rangle : H\in\mathcal{H}\},
\end{equation*}
for some $\rho\in C_{\Aut P}(P')$ in which $\mathcal{H}\rho=\mathcal{H}$ and
$\rho^2=1$.
Furthermore, $C_{\Aut P}(P')$ acts transitively on the set of semi-refined central decompositions.
\end{coro}
\begin{proof}
Let $\mathcal{H}$ be fully refined central decomposition of $P$.  

As $P'=Z(P)$, $b:=\Bi(P)$ is non-degenerate.  Let $\mathcal{X}:=\Bi(\mathcal{H})$ (cf. Section \ref{sec:bi}).  By \thmref{thm:main-reduction}.$(i)$ we know $\mathcal{X}$ is a fully refined
$\perp$-decomposition of $b$.  By \thmref{thm:semi-refined} there is an
isometry $\rho$ which permutes $\mathcal{X}$ such that $\mathcal{X}^{[\rho]}$ 
is semi-refined.  Let $\tau$ be the automorphism on $\mathcal{H}$ induced by 
$\rho$ (cf. \propref{prop:auto-isom}).  Thus, $H\tau\neq H$ only if $H$ is centrally indecomposable of orthogonal type (see \defref{def:semi-refined} and \thmref{thm:indecomps}) and $H@=H\tau @$ (cf. 
\corref{coro:isom-address}.$(iii)$).  Hence, 
$\langle H,H\tau\rangle\cong H\circ H$ for each $H\neq H\tau$, $H\in\mathcal{H}$.  This makes $\mathcal{H}^{[\tau]}$ semi-refined.

Given any other fully refined central decomposition $\mathcal{K}$ of $P$ 
 it follows that $\mathcal{K}$ can be semi-refined 
by an automorphism $\mu$ which permutes $\mathcal{K}$.
Thus, $\mathcal{H}^{[\tau]}$ and $\mathcal{K}^{[\mu]}$ have full refinements with a common address.  Therefore \corref{coro:main}, 
\thmref{thm:main-reduction}.$(ii.b)$, and \corref{coro:refine} prove the transitivity of $C_{\Aut P}(P')$.
\end{proof}

\begin{proof}[Proof of \thmref{thm:central-count}.$(i)$]
First assume that $P'=Z(P)$.  By \thmref{thm:central-III} we know all fully refined 
central decomposition have the same size.  By \corref{coro:central-II}, we know that 
all semi-refinements of
a fully refined central decomposition are equivalent under $\Aut P$.  Furthermore,
this also shows that a semi-refined central decomposition has the form
$\mathcal{H}^{[\rho]}=\{\langle H,H\rho\rangle : H\in\mathcal{H}\}$
where $\rho\in \Aut P$.  Therefore the multiset 
$\{|\langle H,H\rho\rangle| : H\in\mathcal{H}\}$
is uniquely determined by $P$.  Indeed, $|H|=|\langle H,H\rho\rangle|/[H:Z(H)]$ is 
uniquely determined by $H$ and $P$.  

Choose a subset $\mathcal{K}$ of $\mathcal{H}$ in which $H\in\mathcal{K}$ if,
and only if, $H\rho\notin \mathcal{K}$, for each $H\in\mathcal{H}$ with $H\neq H\rho$.  Then $\mathcal{H}$ is 
partitioned into
\begin{equation}
	\{H\in\mathcal{H} : H\rho=H\}\sqcup \mathcal{K}\sqcup\mathcal{K}\rho.
\end{equation}
Hence, the multiset $\{|H|:H\in\mathcal{H}\}$ equals
\begin{equation}
	\{|H|: H\in\mathcal{H} : H\rho=H\}\sqcup 
		\{|K|:K\in\mathcal{K}\}\sqcup\{|K|: K\in\mathcal{K}\}.
\end{equation}
Thus, the multiset of orders of members of $\mathcal{H}$ is uniquely determined by $P$.  

The similar argument works for the multiset of orders of the centers of the members
of $\mathcal{H}$.

Finally, for the case when $P'<Z(P)$ we invoke \lemref{lem:contract-2} and
\lemref{lem:contract-1}.$(ii)$.
\end{proof}

%
%
\section{Unbounded numbers of orbits of central decompositions}\label{sec:ex}

The proof of \thmref{thm:central-count}.$(i)$ has depended on
a study of $C_{\Aut P}(P')$.  Whenever $C_{\Aut P}(P')$ is transitive on the set 
of fully refined central decompositions (\thmref{thm:main-reduction} and \corref{coro:main}) this approach is sufficient.  However, $C_{\Aut P}(P')$ may have multiple orbits.  
This occurs only if there are centrally indecomposable $p$-groups of orthogonal type 
(cf. \thmref{thm:indecomps}).

In this section we have two principal aims: first to show one way that symmetric bilinear 
forms arise in the context of $p$-groups.  Secondly,
we develop examples of centrally indecomposable $p$-groups of the other types specified in \thmref{thm:indecomps}, with the exception of the unitary type.

Most the constructions and theorems in this section are subsumed by more general 
results in \cite{Wilson:grp-alge}, but the proofs provided here are self-contained and 
require fewer preliminaries.

		\subsection{Centrally indecomposable $p$-groups of orthogonal type}

In \cite{Wilson:grp-alge} we prove that there are exponentially many $p$-groups
of order $p^n$ which have class 2, exponent $p$, and are centrally indecomposable of 
orthogonal type.
Indeed, we also show that a $p$-group of class $2$ and exponent $p$ 
with ``randomly selected presentation'' is ``almost always'' a centrally indecomposable 
group of orthogonal type.  Here we describe just one family of centrally indecomposable 
$p$-groups of orthogonal type.

\begin{lemma}\label{lem:free-alt-adj}
Let $V$ be a $k$-vector space of dimension $n>2$.  Define 
$b:V\times V\to V\wedge V$ by $b(u,v):=u\wedge v$, for all $u,v\in V$.
Then $b$ is alternating and $\Adj(b)\cong k$ with 
trivial involution, that is, $b$ is $\perp$-indecomposable of orthogonal type.
\end{lemma}
\begin{proof}
Take $g\in\Adj(b)$.  We show that $g$ is a scalar matrix and thus $\Adj(b)\cong k$.  Hence $b$ is $\perp$-indecomposable of type $1$ with respect to $k$.

Let $V=\langle e_1,\dots, e_n\rangle$ so that $\{e_i\wedge e_j: 1\leq i<j\leq n\}$ is 
a basis of $V\wedge V$.  Fix $1\leq i,j\leq n$, $i\neq j$.  We have
\begin{equation}\label{eq:part1}
	e_i g\wedge e_j=b(e_i g,e_j)=b(e_i,e_j g^*)=e_i\wedge e_j g^*,
		\qquad 1\leq i<j\leq n.
\end{equation}
If we take $e_i g =\sum_{s=1}^n g_{is} e_s$ and $e_j g^*=\sum_{t=1}^n g^*_{jt} e_t$, then
\begin{align*}
0 & =e_i g\wedge e_j-e_i\wedge e_j g^*
	= \sum_{s=1}^n g_{is} (e_s\wedge e_j) - \sum_{t=1}^n g_{jt}^* (e_i\wedge e_t)\\
	& = \sum_{s=1, s\neq i}^n g_{is} (e_s\wedge e_j) + (g_{ii}-g^*_{jj})e_i\wedge e_j
		- \sum_{t=1, t\neq j}^n g_{jt}^* (e_i\wedge e_t).
\end{align*}
So we have $g_{is}=0$ for all $s\neq i$ and $g^*_{jt}=0$ for all $t\neq j$, $1\leq s,t\leq n$
and furthermore $g_{ii}=g_{jj}^*$.  As this is done for arbitrary $1\leq i,j\leq n$, $i\neq j$, we have $g_{11}=g_{22}^*=g_{ii}$ for all $2<i\leq n$.  Finally, $g_{22}=g_{11}^*=g_{33}=g_{11}$
so in fact $g=g_{11} I_n$ and similarly $g^*=g_{11}I_n$.  As $g$ was arbitrary, 
$\Adj(b)=k$.
\end{proof}

If $\dim V=2$ then $V\wedge V\cong k$ and the 
$k$-bilinear map $b$ is simply the non-degenerate alternating 
$k$-bilinear form of dimension $2$.  This is indecomposable 
of symplectic type (\lemref{lem:alt-form}) and the corresponding group is the extra-special
group of order $p^3$ and exponent $p$.

\begin{coro}\label{coro:free-ortho}
Let $V$ be an $\mathbb{F}_q$-vector space of dimension $n>2$ and let 
$b:V\times V\to V\wedge V$ be defined by $b(u,v)=u\wedge v$ for all $u,v\in V$.  
Then $\Grp(b)$ is centrally indecomposable of orthogonal type (see Section \ref{sec:grp}).  
\end{coro}
\begin{proof} This follows from \thmref{thm:main-reduction}.\end{proof}

When $q=p$, $\Grp(b)\cong\langle a_1,\dots,a_n \mid $ 
class $2$, exponent $p\rangle$.
Note that the smallest example of an orthogonal type group is $\langle a_1, a_2,a_3\mid $ class $2$,
exponent $p\rangle$ -- the free class 2 exponent $p$-group of rank 3 and order $p^6$.

		%
		%
		\subsection{Direct sums and tensor products}

Direct sums and tensor products are two natural means to construct bilinear maps from 
others.
To use these we must demonstrate that the adjoint algebras of such products are determined
by the adjoints of the components.  A full account is given in \cite{Wilson:grp-alge} but
here we give only the cases required and supply direct computational proofs.

\begin{defn}
Let $b:V\times V\to W$ and $b':V'\times V'\to W'$ be $k$-bilinear maps.
Let $b\oplus b':V\oplus V'\times V\oplus V'\to W\oplus W'$ be the
bilinear map defined by
	\[(b\oplus b')(u\oplus u',v\oplus v')=b(u,v)\oplus b'(u',v')\]
for all $u,v\in V$ and $u',v'\in V'$.
\end{defn}

\begin{prop}
Let $b$ and $b'$ be two non-degenerate bilinear maps.
Then $\Adj(b\oplus b')=\Adj(b)\oplus \Adj(b')$, where the $*$-operator on the right
hand side is componentwise.   Hence also $\Sym(b\oplus b')=\Sym(b)\oplus \Sym(b')$.
\end{prop}
\begin{proof}
Evidently $\Adj(b)\oplus \Adj(b')\leq \Adj(b\oplus b')$.  For the reverse,
let $f\in \Adj(b\oplus b')\in \End (V\oplus V')$.  
Given $u,v\in V$, $v'\in V'$, take $(u\oplus 0)f=x\oplus x'$ and $(v\oplus v')f=y\oplus y'$
for some $x\oplus x',y\oplus y'\in V\oplus V'$.  It follows that
\begin{multline*}
b(x,v)\oplus b'(x',v')
	=(b\oplus b')((u\oplus 0)f,v\oplus v')\\
	=(b\oplus b')(u\oplus 0,(v\oplus v')f^*)
	=b(u,y)\oplus b'(0,y')
	=b(u,y)\oplus 0.
\end{multline*}
Therefore $b'(x',v')=0$ for all $v'\in V'$.  So $x'\in \rad b'=0$.  Thus 
$(u\oplus 0)f\in V\oplus 0$ for all $u\in V$.   Similarly $(0\oplus v')f\in 0\oplus V'$.  
So $f\in (\End V)\oplus(\End V')$.  

Let $f=g\oplus h$ and $f^*=g^*\oplus h^*$ 
for $g,g^*\in \End V$ and $h,h^*\in \End V'$.  It follows that
\begin{multline*}
b(ug,v)\oplus 0
	=b(ug,v)\oplus b'(u',0)
	=(b\oplus b')((u\oplus u')f,v\oplus 0)\\
	=(b\oplus b')(u\oplus u',(v\oplus 0)f^*)
	=b(u,vg^*)\oplus b'(u',0).
\end{multline*}
Therefore $g\in \Adj(b)$ and similarly $h\in\Adj(b')$.  
So $f\in\Adj(b)\oplus\Adj(b')$.
\end{proof}

Given two bilinear maps $b:V\times V\to W$ and $b':V'\times V'\to W'$
we induce a multi-linear map $(b\times b'):V\times V'\times V\times V'\to W\otimes W'$ defined by:
\begin{equation}
	(b\otimes b') (u, u',v,v')
		:= b(u,v)\otimes b'(u',v'),\qquad \forall u,v\in V, u',v'\in V'.
\end{equation}
Let $\widehat{b\times b'}$
denote the induced linear map $V\otimes V'\otimes V\otimes V'\to W\otimes W'$.  
With this notation we give:
\begin{defn}
Let $b\otimes b':V\otimes V'\times V\otimes V'\to W\otimes W'$ be the
restriction of $\widehat{b\times b'}$ to $V\otimes V'\times V\otimes V'$,
where $b:V\times V\to W$ and $b':V'\times V'\to W'$ are bilinear maps.
\end{defn}
Evidently, $b\otimes b'$ is bilinear.  Using tensor products and the following obvious 
result, we can convert symmetric bilinear maps to alternating bilinear maps.
\begin{prop}
Let $b:U\times U\to W$ and $c:V\times V\to X$ be Hermitian maps over $k$
with involutions $\theta$ and $\tau$, respectively.  Then $b\otimes c$ is
Hermitian with involution $\theta\otimes \tau$.  In particular, 
the tensor of two symmetric bilinear maps is symmetric, the tensor of a
symmetric and an alternating bilinear map is alternating, and the tensor
of two alternating bilinear maps is symmetric.
\end{prop}

\begin{prop}\label{prop:tensor}
Let $d:U\times U\to C$ be a non-degenerate Hermitian $C$-form with 
$k=\{x\in C:\bar{x}=x\}$
and let $b':V\times V\to W$ be a $k$-bilinear map.  Then 
$\Adj(d\otimes b)=\Adj(d)\otimes \Adj(b)$
and $\Sym(d\otimes b)=\Sym(d)\otimes \Sym(b)$.
\end{prop}
\begin{proof}
Clearly $\Adj(d)\otimes \Adj(b)\leq \Adj(d\otimes b)$.
For the reverse inclusion, let $\mathcal{X}$ be an orthogonal basis of $d$ 
and $\mathcal{E}=\mathcal{E}(X)$.  Take $g\in \Adj(d\otimes b)$.  We
show that $g\in \Adj(d)\otimes \Adj(b)$.

If $x,y\in \mathcal{X}$ with associated idempotents $e,f\in\mathcal{E}$,
then $(e\otimes 1)g(f\otimes 1)$ restricts to 
$\langle x\rangle\otimes V\to \langle y\rangle \otimes V$, so
there is a $g_{x,y}:V\to V$ defined by $vg_{x,y}=v'$, where 
$(x\otimes v)(e\otimes 1)g(f\otimes 1)=y\otimes v'$.  
Let $(x,y)$ be the transposition interchanging $x$ and $y$ and identity
on $\mathcal{X}-\{x,y\}$, treated as an element of $\End U=\Adj(d)$.
Set $e_{x,y}=e(x,y)f$.
Thus, $(e\otimes 1)g(f\otimes 1)=e_{x,y} \otimes g_{x,y}$.
Since
\begin{equation*}
g = \left(\sum_{e\in \mathcal{E}} e\otimes 1\right)g
	\left(\sum_{f\in\mathcal{E}} f\otimes 1\right)
	= \sum_{e,f\in\mathcal{E}} (e\otimes 1)g(f\otimes 1)\\
	= \sum_{x,y\in\mathcal{X}} e_{x,y}\otimes g_{x,y},
\end{equation*}
it suffices to prove that $g_{x,y}\in \Adj(b)$.

As $(e\otimes 1)g(f\otimes 1)\in\Adj(d\otimes b)$ with
$((e\otimes 1)g(f\otimes 1))^*=(f\otimes 1)g^*(e\otimes 1)$ it follows that:
\begin{align*}
1\otimes b(v (d(y,y)g_{x,y}),v')
	& = d(y,y)\otimes b(vg_{x,y},v')\\
	& = (d\otimes b)((x\otimes v)(e\otimes 1)g(f\otimes 1),y\otimes v')\\
	& = (d\otimes b)(x\otimes v,(y\otimes v')(f\otimes 1)g^*(e\otimes 1))\\
	&= d(x,x)\otimes b(v,v'g^*_{y,x})
	= 1 \otimes b(v,v'(d(x,x) g^*_{y,x}) ).
\end{align*}
Notice we have used the fact that $d(x,x),d(y,y)\in k^\times$ and that the
tensor is over $k$.  Therefore 
$b(v g_{x,y},v')=\frac{d(x,x)}{d(y,y)}b(v,v'g^*_{y,x}) $ for all $v,v'\in V$.
Hence $g_{x,y}\in\Adj(b)$ with adjoint $\frac{d(x,x)}{d(y,y)} g^*_{y,x}$.
This completes the proof.
\end{proof}
It can be shown that $\Adj(b\otimes c)=\Adj(b)\otimes\Adj(c)$
for any two bilinear maps $b$ and $c$ \cite{Wilson:grp-alge}.

		%
		%
\subsection{Proof of \thmref{thm:central-count}.$(ii)$}
		\label{sec:orthog-ex}

The best known examples of central products are the extra-special $p$-groups of
exponent $p$ and rank $2n$: $p^{1+2n}=\overbrace{p^{1+2}\circ\cdots \circ p^{1+2}}^n$, 
\cite[Theorem 5.5.2]{gor}.  It is customary to recognize $\Bi(p^{1+2n})$ as
the non-degenerate alternating bilinear form 
$b:\mathbb{Z}_p^{2n}\times \mathbb{Z}_p^{2n}\to \mathbb{Z}_p$.
We view this map as $d\otimes c$,
where $d:\mathbb{Z}_p^n\times \mathbb{Z}_p^n\to \mathbb{Z}_p$ is the
dot product $d(u,v):=uv^t$, and 
$c:\mathbb{Z}_p^2\times \mathbb{Z}_p^2\to \mathbb{Z}_p\wedge\mathbb{Z}_p$.  

This construction generalizes.  If $H$ is a centrally indecomposable group then 
it has an associated associative composition algebra $C:=\Adj(\Bi(H))/\rad \Adj(\Bi(H))$
(cf. \thmref{thm:indecomps}).  Recall that $K:=\{x\in C: x=\bar{x}\}$ is a field (cf.
\defref{def:comp}).  Set $P:=\overbrace{H\circ \cdots \circ H}^n$ and
$b:=\Bi(\overbrace{H\circ\cdots \circ H}^n)$.  As in \exref{ex:cent-perp}, 
$b=\overbrace{\Bi(H)\perp\cdots\perp \Bi(H)}^n$ which we can express compactly
as $b=d\otimes_K \Bi(H)$, where $d:K^n\times K^n\to K$ is
the usual dot product $d(u,v):=uv^t$, $u,v\in K^n$.  Hence, by \propref{prop:tensor},
it follows that $\Adj(b)=\Adj(d)\otimes_K \Adj(\Bi(H))$ and thus
$\Adj(b)/\rad \Adj(b)\cong \Adj(d)\otimes_K C$.   Yet, $\Adj(d)\otimes_K C\cong \Adj(d')$, where $d':C^n\times C^n\to C$ is defined by $d'(u,v):=u\bar{v}^t$, for $u,v\in C^n$.
If $C> K$ then \corref{coro:main} proves that all fully refined central decompositions of $P$
are conjugate under automorphisms of $P$.   We now demonstrate that the same is not generally possible with orthogonal type.

\begin{lemma}\label{lem:many-addresses}
Let $H=\langle X\rangle$ be a centrally indecomposable $p$-group of orthogonal 
type over $\mathbb{F}_q$ with $X$ a minimal generating set of $H$.
Set $P:=\overbrace{H\circ\cdots\circ H}^n$ and let $\mathcal{H}_0=\{H_1,\dots, H_n\}$ be the canonical central decomposition given by the central product,
so that $H_i=\langle x_i : x\in X\rangle$ where $x_i$ denotes $x$ in the $i$-th
component.  

Let $\omega=\alpha^2+\beta^2\in\mathbb{Z}_p$ be a non-square.
If $0\leq m\leq n/2$ then define 
$\mathcal{H}_m=\{K_1,\dots,K_{2m},H_{2m+1},\dots,H_n\}$ where
	\[K_{2j-1} :=\langle x_{2j-1}^{\alpha} x_{2j}^{\beta}: x\in X \rangle,\quad
		K_{2j} :=\langle x_{2j-1}^{\beta} x_{2j}^{-\alpha}: x\in X \rangle,\]
for $1\leq j\leq m$.   Then every member of $\mathcal{H}_m$ is isomorphic
 to $H$ and 
$\mathcal{H}_m$ is a fully refined central decompositions
of $P$ with address $(n-2m:2m)$, for $1\leq m\leq n/2$.
\end{lemma}
\begin{proof}
As $X$ is a minimal generating set of $H$, if $x,y\in X$ with $Z(H)x=Z(H)y$ then $x=y$.
Therefore, $\overbrace{X\times\cdots \times X}^n$ is mapped injectively into $P$ via
the homomorphism $\pi:\prod_{H\in\mathcal{H}} H\to P$ described in Section \ref{sec:cent-prod}. 
This makes the groups $H_i$, $K_{2j-1}$, and $K_{2j}$ well-defined,
for each $1\leq i\leq n$ and $1\leq j\leq n/2$.  Furthermore, $H_i \cong H$ for each
$1\leq i\leq n$ and $\mathcal{H}_0$ is a fully refined central decomposition of $P$.

Set $X_i=H_i/H'_i=H_i P'/P'$, $W=P'=H'_i$, $1\leq i\leq n$.  Also set 
$L_j:=\langle H_{2j-1},H_{2j}\rangle=\langle K_{2j-1},K_{2j}\rangle$, $1\leq j\leq n/2$.
Then $L_j/L'_j=X_{2j-1}\oplus X_{2j}$ and $b|_{L_j/L'_j}$ is $b\perp b$ where $b=\Bi(H)$.
Recall that $\Bi(P)=d\otimes b$ where $d:k^n\times k^n\to k$ is the dot product
and $\mathcal{X}:=\Bi(\mathcal{H}_0)=\{X_i : 1\leq i\leq n\}$ is a fully refined $\perp$-decomposition of $\Bi(P)$.  As $\Adj(\Bi(P))=\Adj(d)\otimes \Adj(\Bi(H))\cong \Adj(d)$, it follows that 
$\mathcal{X}_d=\{Y_1,\dots, Y_n\}$ is fully refined $\perp$-decomposition of $d$.
In fact, the implied isomorphism $\Adj(\Bi(P))$ to $\Adj(d)$ maps $f\otimes 1\to f$,
so $\mathcal{E}(\mathcal{X})$ is sent to the canonical frame
$\{\Diag\{1,0,\dots\},\dots,\Diag\{\dots,0,1\}\}$ of $\Adj(d)$.  So,
$\mathcal{H}_0@=\mathcal{X}_d @=(n:0)$.

Define
\begin{equation*}
	(\varphi_j,\hat{\varphi}_j)
		:=\left(\begin{bmatrix} \alpha 1_{X_{2j-1}} & \beta 1_{X_{2j}}\\
								\beta 1_{X_{2j-1}} & -\alpha 1_{X_{2j}}\end{bmatrix},
					(\alpha^2+\beta^2)1_W\right)\in \Isom^* (b|_{L_j/L'_j}).
\end{equation*}
Set $\tau_j := \Grp(\varphi_j,\hat{\varphi}_j)\in \Aut L_j$. 
Then $K_{2j-1}=H_{2j-1}\tau_j$ and $K_{2j}= H_{2j}\tau_j$ for $1\leq j\leq n/2$.  
Furthermore, $(\varphi_j,\hat{\varphi}_j)$ induces
\begin{equation*}
	\left(\begin{bmatrix} \alpha  & \beta \\
								\beta & -\alpha \end{bmatrix},
					\omega \right)\in \Isom^* (\langle Y_{2j-1}, Y_{2j}\rangle).
\end{equation*}
Therefore, $K_{2j-1}@=[\omega]$ and $K_{2j}@=[\omega]$.
Thus we have proved that $\mathcal{H}_m$ has address $(n-2m:2m)$.
\end{proof}

At this point we know there are multiple $C_{\Aut P}(P')$-orbits of fully refined central
decompositions of $P$, for any $P$ satisfying the hypothesis of \lemref{lem:many-addresses}.
But we have not worked with $\Aut P$-orbits yet.  We now show that 
there are multiple $\Aut P$-orbits as well.

\begin{lemma}\label{lem:tensor}
Given vector spaces $U$ and $V$, the map $\alpha\oplus \beta\to \alpha\otimes \beta$
from $\GL(U)\oplus \GL(V)\to\GL(U\otimes V)$ has kernel
\begin{equation*}
	Z := \langle s1_U\oplus s^{-1}1_V \mid s\in k^\times \rangle.
\end{equation*}
and the image is isomorphic to $\GL(U)\circ \GL(V) = (\GL(U)\oplus \GL(V))/Z$.
\end{lemma}
\begin{proof} Fix bases for $U$ and $V$ and consider the resulting
matrices in the kernel. \end{proof}

\begin{thm}\label{thm:many-orbits}
Let $H:=\langle x,y,z | $ class $2$, exponent $p\rangle$ (which is centrally indecomposable
by \corref{coro:free-ortho}), $P:=\overbrace{H\circ\cdots\circ H}^{2n}$ and 
$\mathcal{H}_m$ be as in \lemref{lem:many-addresses}.  Then all the following hold:
\begin{enumerate}[(i)]
\item Every member of $\mathcal{H}_m$ is
isomorphic to $H$.
\item $\mathcal{H}_m$ is a fully refined central decomposition of $P$.
\item For every fully refined central decomposition $\mathcal{K}$ of $P$, there is
is a unique $m$ and some $\alpha\in C_{\Aut P}(P')$ such that 
$\mathcal{K}^{\alpha}=\mathcal{H}_m$.  So there are $1+n$ orbits 
of fully refined central decomposition under the action of $C_{\Aut P}(P')$.
\item $\mathcal{H}_m$ and 
$\mathcal{H}_{m'}$ are in the same $\Aut P$-orbit if, and only if, $m'=n-m$.  
\end{enumerate}
Hence there are exactly $1+\lfloor n/2\rfloor$ orbits in the set of fully 
refined central decompositions of $P$ under the action of $\Aut P$.
\end{thm}
\begin{proof}
Let $k:=\mathbb{Z}_p$.

By design, $\Bi(H)$ is the map $c:V\times V\to W$ where $V=k^3$, 
$W:=k^3\wedge k^3\cong k^3$ and $c(u,v)=u\wedge v$, 
$u,v\in V$.  Hence $(i)$ and $(ii)$ follow from \lemref{lem:many-addresses}.  Furthermore, 
every possible address (see \corref{coro:count}) of $\Bi(P)$ is given by one of 
the $\mathcal{H}_m$.   Therefore $(iii)$ follows from \corref{coro:main} and 
\thmref{thm:main-reduction}.

To prove $(iv)$ we start by describing the structure of $\Isom^*(b)$.  
Set $b:=\Bi(P)$ and recall that $b=d\otimes c$ where $d:U\times U\to k$ is 
the dot product on $U:=k^n$.  Following \lemref{lem:tensor} we find that 
\begin{equation*}
	\Isom^*(d)\circ \Isom^*(c)=\Isom^*(d)\oplus \Isom^*(c)/
		\langle (s 1_U\oplus s^{-1} 1_V, 1_{k\otimes W}): s\in k^\times\rangle
\end{equation*}
embeds in $\Isom^*(b)$.  \emph{We claim that $\Isom^*(b)$ equals this embedding}.

By \propref{prop:tensor} we know that $\Adj(b)=\Adj(d)\otimes\Adj(c)\cong \Adj(d)$.  
Hence $\Isom(b)\cong \Isom(d)=\GO(d)$.  Indeed this shows that
\begin{equation*}
	\Isom(b)=\{\alpha\otimes 1_V : \alpha\in \GO(d)\}.
\end{equation*}
In particular, $\Isom(b)$ embeds in $\Isom^*(d)\circ\Isom^*(c)$.  Observe that
$\Isom^*(c)=\{(f,f\otimes f) : f\in\GL(V)\}\cong \GL(3,p)$ and use 
\lemref{prop:dot-isom-*} to conclude that
\begin{align*}
	[\Isom^*(d)\circ \Isom^*(c):\Isom(b)]
		& = \frac{ (p-1)|\GO(d)| |\GL(3,p)|}{(p-1) |\GO(d)|}
		= |\GL(3,p)|.
\end{align*}

As $\Isom^*(b)/\Isom(b)\leq \GL(k\otimes W)\cong \GL(3,p)$, we conclude 
by orders that $\Isom^*(b)=\Isom^*(d)\circ \Isom^*(c)$. Hence the 
orbits of $\Isom^*(b)$ on fully refined central decompositions are those of 
$\Isom^*(d)\circ\Isom^*(c)$, that is, the orbits described in
\corref{coro:pseudo-orbits}.
\end{proof}

\begin{proof}[\thmref{thm:central-count}.$(ii)$]
This follows from \thmref{thm:many-orbits}.
\end{proof}
		\subsection{Centrally indecomposable $p$-groups of non-orthogonal type}

Centrally indecomposable families of symplectic type are the easiest to construct by
classical methods.  Already the extraspecial $p$-groups $p^{1+2}$ of exponent $p$
serve as examples.  We generalize the extraspecial example to include
field extensions of $\mathbb{Z}_p$.  We let $k$ be an arbitrary field.

\begin{lemma}\label{lem:alt-form}
The $k$-bilinear form $d:k^2\times k^2\to k$ 
defined by $d(u,v)=\det\begin{bmatrix} u\\ v\end{bmatrix}$, for all 
$u,v\in k^2$, has $\Adj(d)=M_2(k)$ with the 
adjugate involution, thus $d$ is $\perp$-indecomposable of symplectic type.
\end{lemma}

\begin{coro}
Let $d:\mathbb{F}_q^2\times \mathbb{F}_q^2\to \mathbb{F}_q$ be the non-degenerate
alternating bilinear form of dimension 2.  Then $\Grp(d)$ is centrally indecomposable 
of symplectic type.
\end{coro}
\begin{proof} This follows from \thmref{thm:main-reduction}.\end{proof}

Presently we are not aware of any alternating bilinear maps which are centrally
indecomposable of unitary type.  We expect infinite families over any field $\mathbb{F}_{q^2}$
to exist.  Our search for such examples is on-going.

We next construct a family of centrally indecomposable $p$-groups of exchange type.
This family also illustrates that there can be a non-trivial $O_p(C_{\Aut P}(P'))$.
There are families of exchange type without this feature but we choose this family
for the ease of proof. 

\begin{lemma}\label{lem:exchange}
Let $V$ be a $k$-vector space of dimension $n>1$.  Define the $k$-bilinear
map $b:(k\oplus V)\times (k\oplus V)\to V$ by
\begin{equation}\label{eq:exchange}
	b(\alpha\oplus u,\beta\oplus v) := \alpha v-\beta u.
\end{equation}
Then $b$ is alternating and 
\begin{equation*}
\Adj(b)   \cong  
\left\{\begin{bmatrix} \alpha 1_k & h \\ 0 & \beta 1_V\end{bmatrix}:
 h\in\hom(k,V), \alpha,\beta\in k\right\},
\end{equation*}
where the multiplication and the action on $k\oplus V$ is interpreted as 
matrix multiplication and where the involution is defined by
\begin{equation*}
\begin{bmatrix} \alpha 1_k & h \\ 0 & \beta 1_V\end{bmatrix}^*
 :=  \begin{bmatrix} \beta 1_k & -h \\ 0 & \alpha 1_V \end{bmatrix}.
\end{equation*}
In particular, $\Adj(b)/\rad \Adj(b)\cong k\oplus k$ with the exchange
involution and the radical is $\left\{\begin{bmatrix} 0 & h\\ 0 & 0\end{bmatrix}
: h\in\hom(k,V)\right\}$.  Thus $b$ is $\perp$-indecomposable of exchange type.
\end{lemma}
\begin{proof}
It is easily checked that $e:=1_k\oplus 0_V, f:=0_k\oplus 1_V\in 
\End (k\oplus V)$ are both in $\Adj(b)$ and furthermore $e^*=f$, $e^2=e$, $f^2=f$.  Fix 
$g\in \Adj(b)$.  Then $ege$, $egf$, $fge$ and $fgf$ lie in $\Adj(b)$ and
$g=ege+egf+fge+fgf$.  Therefore, we need only characterize these four terms.

Let $u,v\in V$ be linearly independent.  Since $(0\oplus u)fge=\lambda\oplus 0$
and $(0\oplus v)fg^*e=\tau\oplus 0$ for some $\lambda,\tau\in k$, it follows that
\begin{equation*}
\begin{split}
\lambda v
& =b(\lambda\oplus 0,0\oplus v)
=b((0\oplus u)g,0\oplus v)=b(0\oplus u,(0\oplus v) g^*)\\
& =b(0\oplus u,\tau\oplus 0)
=-\tau u.
\end{split}
\end{equation*}
However, $u$ and $v$ are linearly independent, and hence $\lambda=0=\tau$ so $fge=0=fg^*e$.

Next let $(1\oplus 0)ege=\alpha\oplus 0$ and $(0\oplus u)fg^*f=0\oplus v$ for some 
$\alpha\in k$ and $v\in V$.  Then
\begin{equation*}
\begin{split}
\alpha u
& =b(\alpha\oplus 0,0\oplus u)
=b((1\oplus 0)ege,0\oplus u)=b(1\oplus 0,(0\oplus u)g^*)\\
& =b(1\oplus 0,0\oplus v)
=v.
\end{split}
\end{equation*}
Thus $fg^*f=0\oplus \alpha 1_V$ where $ege=\alpha 1_k\oplus 0_V$.
Setting $(0\oplus u)fgf =0\oplus v$ and $(1\oplus 0)eg^*e=\beta\oplus 0$ we similarly find
$fgf=0\oplus \beta 1_V$ where $eg^* e=\beta 1_k\oplus 0_V$.

Finally, set $(1\oplus 0)egf=0\oplus u$ and $(1\oplus 0)eg^* f=0\oplus v$.  Then
\begin{equation*}
\begin{split}
-u
& =b(0\oplus u,1\oplus 0)
=b((1\oplus 0)egf,1\oplus 0)=b(1\oplus 0,(1\oplus 0)eg^*f)\\
& =b(1\oplus 0,0\oplus v)
=v.
\end{split}
\end{equation*}
So $egf$ is induced by a $k$-linear map $h:k\to V$ and $eg^*f$ is
induced by $-h$.
\end{proof}

\begin{coro}
Let $b:(\mathbb{F}_q\oplus \mathbb{F}_q^n)\times 
(\mathbb{F}_q\oplus \mathbb{F}_q^n)
\to \mathbb{F}_q^n$ be as in \eqref{eq:exchange} with $n>1$.   Then 
$\Grp(b)$ is centrally indecomposable of exchange type.
\end{coro}
\begin{proof} This follows from \thmref{thm:main-reduction}.\end{proof}

If $n=1$ then $b$ is simply the non-degenerate alternating 
bilinear $k$-form of dimension $2$.  The smallest example 
of a $p$-group with exchange type is in fact of order $p^5$ with rank $3$,
namely: $\langle a,b,c \mid [a,c]=1,$ class $2$, exponent $p\rangle$.

We can \eqref{eq:exchange} to illustrate that Section \ref{sec:radicals} is required.
We emphasize that instances of non-trivial radicals are known in far more 
general settings than $\perp$-indecomposable bilinear maps of exchange type.  

The radical in of $\Adj(b)$, for $b$ as in \eqref{eq:exchange}, intersects
$\Sym(b)$ trivially.  However, if we define 
$c:(k\oplus V)\times (k\oplus V)\to V$ by
\begin{equation}
	c(\alpha\oplus u, \beta\oplus v) := \alpha v + \beta u,\qquad 
		\forall \alpha,\beta,\in k, u,v\in V;
\end{equation}
then $\Adj(c)/\rad \Adj(c)\cong k\oplus k$ with the exchange involution.
Here $\rad\Adj(c)\leq \Sym(c)$.  To make this example alternating we may
simply tensor by the alternating bilinear map from \lemref{lem:free-alt-adj}.
To further make a $\perp$-decomposable bilinear map we may tensor with
a dot-product.  By \propref{prop:tensor}, the result has a non-trivial radical 
in $\Sym(b)$.

\section{Closing remarks}\label{sec:closing}

		%
		%
\subsection{Conjecture on uniqueness of fully refined central decompositions}
		\label{sec:conjecture}
		
It remains open whether or not the multiset of isomorphism types of a
fully refined central decomposition of a $p$-group $P$ of class $2$ and exponent $p$
is uniquely determined.  By \corref{coro:central-II}, it suffices to answer the following:
\begin{quote}
Let $H$ and $K$ be centrally indecomposable $p$-groups of class $2$, exponent $p$,
and of orthogonal type.  Is true that whenever $H\circ H\cong K\circ K$ then 
$H\cong K$? (Where $H\circ H$ is as in \eqref{eq:canonical-cent}).
\end{quote}
We conjecture this is true.  

A counter-example will require a $p$-group $P:=H\circ H$ of class $2$ and 
exponent $p$, where $H$ is non-abelian centrally indecomposable and  
$\Adj(\Bi(P))/\rad \Adj(\Bi(P))\cong \Adj(d)$ with $d$ the dot-product
$d(u,v)=uv^t$, $u,v\in K^2$.  Furthermore, $\Bi(P)$ must not be bilinear 
over $K$.  Infinite families of this sort are known, but so far have not produced 
counter-examples.

Because such groups involve symmetric bilinear forms, it is possible that a solution
will divide along the congruence of $p$ modulo $4$.  Some evidence of this has been
uncovered while attempting to develop counter-examples.  From these conditions
and others, it appears that a counter-example group $P$ will have order at least 
$5^{30}$.

\subsection{Further directions}
In the future, it may become possible to replace various algebraic arguments 
contained here with group theory.
Nonetheless, 
the condition that an endomorphism $f\in \End V$ lies in $\Adj(b)$
(or $\Sym(b)$) is determined by a system of linear equations.  This is the 
source of polynomial time algorithms for computing
central decompositions of $p$-groups found in \cite{Wilson:algo-cent}.  In 
contrast, the equations to determine if $f\in\Isom(b)$ or $\Isom^*(b)$ 
(and hence to determine the automorphism group of a $p$-group) are quadratic and
generally difficult to solve.

The use of bilinear
maps makes many of the results adaptable to other algebraic objects.
For instance, our theorems apply (at least over finite fields) 
to central decompositions of class $2$ nilpotent Lie algebras.  See also 
\cite{Abbasi:3} and \cite[pp. 608-609]{Bond}.  

\subsection{Other fields}
The use of finite fields removed the need to consider 
Hermitian forms over non-commutative division rings in the classification of 
$*$-simple algebras, and consequently also the related simple Jordan algebras 
(\thmref{thm:*-structure} and \thmref{thm:herm-structure}); therefore, 
this assumption affects Sections \ref{sec:simples} and \ref{sec:orbits}.
 Furthermore, as finite 
fields are separable we are able to apply Taft's $*$-algebra version of the 
Wedderburn Principal theorem (\thmref{thm:Taft}) in proving \thmref{thm:indecomps}.  
Evidently our proofs apply also to bilinear maps over algebraically 
closed fields of characteristic not 2.

\subsection{$2$-groups of exponent $4$}
The omission of $2$-groups of exponent $4$ in the proof of 
\thmref{thm:main-reduction} can be relaxed \cite{Wilson:isoclinism}.  The 
known obstacles for $2$-groups of class 2 and 
exponent $4$ derive from the usual complications of symmetric bilinear forms in 
characteristic 2.  We are presently investigating whether or not these are 
indeed the only limitations.

%
%
\subsection*{Acknowledgments}

I am grateful to W. M. Kantor for many helpful discussions and for
suggesting the Jordan product for $\Sym(b)$; to
E. M. Luks and C. R. B. Wright for their comments on earlier drafts; 
and to E. I. Zel'manov and H. P. Petersson for advice on Jordan
algebras.

%
%


\def\cprime{$'$}

\end{document}